\documentclass{article}

\usepackage{ntheorem}
\usepackage{amsmath,amssymb,amsfonts,graphicx,float,bm}
\usepackage{hyperref}

\usepackage{enumitem,multicol}

\usepackage{kbordermatrix}
\usepackage{pbox}

\def\BBox{\rule{2mm}{3mm}}
\def\QED{\hfill$\BBox$}
\newenvironment{proof}
{\begin{rm}\par\smallskip\noindent{\bf Proof.}\quad}{\QED\end{rm}}
\newenvironment{proofidea}
{\begin{rm}\par\smallskip\noindent{\bf Proof idea.}\quad}{\QED\end{rm}}

\DeclareMathOperator{\st}{star}
\DeclareMathOperator{\rank}{rank}

\DeclareMathOperator{\conv}{conv}

\newtheorem{thm}{Theorem}[section]
\newtheorem{lem}[thm]{Lemma}        %
\newtheorem{prop}[thm]{\bfseries Proposition} %

\theorembodyfont{\rmfamily}

\newtheorem{question}[thm]{\bfseries Question}
\theorembodyfont{\rmfamily}
\newtheorem{defn}[thm]{\bfseries Definition}

\newtheorem*{goal}{\bfseries Goal}
\newcommand{\om}[2]{\textrm{OM}{(#1,#2)}} %
\newcommand{\ceil}[1]{\left\lceil {#1} \right\rceil} %
\newcommand{\ffloor}[2]{\left\lfloor{\frac{#1}{#2}}\right\rfloor} %
\newcommand{\PP}{P}

\begin{document}

\title{Enumerating neighborly polytopes and oriented matroids}

\author{
Hiroyuki Miyata 
\\
Graduate School of Information Sciences, \\
Tohoku University, Japan\\
{\tt hmiyata@dais.is.tohoku.ac.jp}
\and Arnau Padrol\\
Institut f\"ur Mathematik\\
Freie Universit\"at Berlin\\
{\tt arnau.padrol@fu-berlin.de}
}

\maketitle

\begin{abstract}
Neighborly polytopes are those that maximize the number of faces in each
dimension among all polytopes with the same number of vertices.
Despite their extremal properties they form a surprisingly rich class of
polytopes, which has been widely studied and is the subject of many open
problems and conjectures.

In this paper, we study the enumeration of neighborly polytopes beyond the
cases that have been computed so far. To this end, we enumerate neighborly
oriented matroids --- a combinatorial abstraction of neighborly
polytopes --- of small rank and corank.
In particular, if we denote by $\om{n}{r}$ the set of all oriented
matroids of rank~$r$ and $n$ elements, we determine all uniform neighborly
oriented matroids in $\om{5}{\leq12}$, $\om{6}{\leq9}$, $\om{7}{\leq11}$
and $\om{9}{\leq12}$ and all possible face lattices of neighborly oriented
matroids in $\om{6}{10}$ and $\om{8}{11}$. Moreover, we classify all
possible face lattices of uniform $2$-neighborly oriented matroids in
$\om{7}{10}$ and $\om{8}{11}$.
Based on the enumeration, we construct many interesting examples and test
open conjectures.
\end{abstract}

\section{Introduction}

The scarcity of examples (and counterexamples) is a central problem in the study of combinatorial properties of convex polytopes. This makes the enumeration of \emph{all} combinatorial types of $d$-polytopes with $n$ vertices a fundamental problem, even for small values of $d$ and $n$. This line of research was already started by Cayley and Kirkmann in the second half of the nineteenth century (see the historical remarks in Chapters~5, 6, 7 and 13 of~\cite{G03}). However, this is a computationally difficult problem, since realizability of $d$-polytopes is polynomially equivalent to the \emph{Existential Theory of the Reals (ETR)}~\cite{M1988,S1991}, and therefore NP-hard, already for $d=4$~\cite{RZ95}.

As a consequence, the computations become intractable already for very small values of $d$ and $n$. For $d$ polytopes with up to $d+3$ vertices there are closed formulas~\cite{F06}. In dimension~$3$, the problem is equivalent to enumerating $3$-connected planar graphs by Steinitz's Theorem, and the precise asymptotic behavior is known~\cite{BW88,T62}. For the remaining cases, the largest classes that have been completely enumerated and classified so far are $4$-polytopes with $8$ vertices and $5$-polytopes with $9$ vertices~\cite{AS85,FMM13} (see the introduction of~\cite{FMM13} for a summary of state of the art). 

This intrinsic difficulty has motivated the focus on the enumeration of smaller and specially interesting families of polytopes, in particular \emph{simplicial}~\cite{ABS80,G03,GS67} and \emph{neighborly} polytopes \cite{A77,AM73,AS73,BG87,BS87,BSt87,F14,G03,M74,S95}. A $d$-polytope is \emph{$k$-neighborly} if every subset of $k$ vertices forms a face, and it is called just \emph{neighborly} if it is $\ffloor{d}{2}$-neighborly. Neighborly polytopes are important for several reasons, among them the \emph{Upper Bound Theorem}, that states that simplicial neighborly polytopes are those that maximize the number of $i$-dimensional faces among all polytopes of fixed dimension and number of vertices~\cite{M70}. They form a very rich family. Actually, the current best lower bounds for the number of combinatorial types of polytopes is attained by neighborly polytopes~\cite{P13}.

In dimension at most~$3$ every polytope is neighborly, and there are explicit formulas for the number of neighborly $d$-polytopes with $d+3$ vertices \cite{AM73,M74}. Moreover, the enumeration of neighborly polytopes is known for $4$ and $6$-dimensional polytopes with up to $10$ vertices, as a result of the combined effort of several researchers during the 70's and 80's \cite{A77,AS73,BG87,BS87,BSt87,G03}. Recently, the enumeration of simplicial neighborly $5$-polytopes with $9$ vertices has been also completed~\cite{F14,FMM13}.

The usual approach to all these results starts by enumerating all possible \emph{oriented matroids} for given parameters, 
which provides a superset containing all possible combinatorial types of polytopes. Oriented matroids are a combinatorial abstraction 
for point configurations and hyperplane arrangements that can be enumerated in a much more efficient way. 
The second step of this process is to decide which of these oriented matroids admit realizations as point configurations. This is a the hardest part, which usually requires 
ad-hoc solutions and case-by-case analysis. Even when restricted to neighborly polytopes, realizability is NP-hard~\cite{AP14}.
The current state of the art in enumeration of oriented matroids can be found in~\cite{wom}. 
It summarizes results of several authors concerning enumeration of uniform and non-uniform oriented matroids and their 
realizability \cite{AAK02,AK07,FF02,FF03,FMM13,GSL89,GP80,G72,G03,R88}. General oriented matroids of ranks $3$ and $4$ 
(and their duals) have been respectively enumerated up to $10$ and $8$ elements, and their realizability is known for up to $9$ and $8$ elements;
uniform matroids are classified for up to $11$ and $9$ elements, and their realizability for up to $11$ and $8$. 
Neighborly oriented matroids of rank $5$ with $11$ elements were enumerated
by Schuchert in 1995~\cite{S95}.

\medskip

In this paper we approach the generation of neighborly oriented matroids. Our computational approach is based on single element extensions, 
as in \cite{BG00,FF02}. However, we first apply a satisfiability (SAT) solver to force neighborliness and reduce the number of candidate signatures 
(see Section~\ref{sec:computation}). The SAT problem is well known to be NP-complete, but can be solved efficiently in practice.
Since Schewe's work~\cite{S07}, SAT solvers have been successfully used in discrete geometry~\cite{BDHS13,BS11,MMIB12}. With the new method, 
we are able to completely enumerate the following new classes (where $\om{r}{n}$ represents the set of all oriented matroids of rank~$r$ with $n$ elements):

\begin{enumerate}
 \item all neighborly oriented matroids in $\om{5}{12}$, $\om{7}{11}$ and $\om{9}{12}$;
 \item all uniform neighborly oriented matroids in $\om{6}{9}$;
 \item all possible face lattices of uniform neighborly oriented matroids in $\om{6}{10}$ and $\om{8}{11}$.
\end{enumerate}
Additionally, we are also able to enumerate:%
\begin{enumerate}[resume]
 \item all possible face lattices of uniform $2$-neighborly oriented matroids in $\om{7}{10}$ and $\om{8}{11}$.
\end{enumerate}

The results of our enumeration are summarized in Table~\ref{tb:summaryneighborly}, that shows the number of neighborly oriented matroids (and the corresponding face lattices) and in Table~\ref{tb:summary2neighborly}, that shows the number of $2$-neighborly oriented matroids. %
In boldface we have stressed the results that were not known before.
The complete database is available at
\begin{center}
{\tt https://sites.google.com/site/hmiyata1984/neighborly\_{}polytopes}.
\end{center}

\begin{table}[htpb]
\begin{center}
{
\renewcommand{\arraystretch}{1.5}
\begin{footnotesize}
\begin{tabular}{c | c | c | c | c | c | c | c | c | c |}
        & n = 5 & n = 6 & n = 7 & n = 8 & n = 9 & n = 10 & n = 11& n = 12\\
 \hline
 r = 5  &  1  (1)   &   1  (1)    &   1  (1)   &    3  (3)    &  23  (23)    &  432  (432) & 13\,937  (13\,937) & \textbf{556\,144  (556\,144)} \\
 \hline
 r = 6  &             &   1  (1)    &    1  (1)   &   2  (2)  &    \textbf{10\,825}  (126)   &  unk.  \textbf{(159\,750)}  & unk. (unk.) & unk. (unk.)\\
 \hline
 r = 7  &       &       &   1  (1)     &   1  (1)  &   1  (1)  &  37  (37)  & \textbf{42\,910  (42\,910)} &  unk. (unk.)\\
 \hline
 r = 8  &       &       &        &  1  (1) &   1  (1)   &  4  (4)  & unk.  \textbf{(35\,993)} &  unk. (unk.)\\
 \hline
 r = 9  &       &       &        &    &   1  (1)  &  1  (1) & 1  (1) &  \textbf{2\,592  (2\,592)} \\
 \hline
\end{tabular}
\end{footnotesize}
}
\end{center}
\caption{The numbers of (relabeling classes of) neighborly uniform oriented matroids of rank~$r$ and $n$ elements (the numbers enclosed by brackets are the number of different face lattices of the corresponding oriented matroids). Boldface results are new.}
\label{tb:summaryneighborly}
\end{table}

\begin{table}[htpb]
\begin{center}
{
\renewcommand{\arraystretch}{1.5}
\begin{footnotesize}
\begin{tabular}{c | c | c | c | c | c | c |}
       & n = 7 & n = 8 & n = 9 & n = 10 & n = 11 \\
\hline
r = 7  &    1  (1)     &   2  (2)  &   \textbf{9  (9)}  &  unk. \textbf{(4\,523)} &  unk. (unk.) \\
\hline
r = 8  &         &  1  (1) &   2  (2)   &  \textbf{13  (13)} & unk. \textbf{(129\,968)}\\
\hline
\end{tabular}
\end{footnotesize}
}
\end{center}
\caption{The numbers of (relabeling classes of)  $2$-neighborly uniform oriented matroids of rank~$r$ and $n$ elements (the numbers enclosed by brackets are the number of different face lattices of the corresponding oriented matroids). Boldface results are new.}
\label{tb:summary2neighborly}
\end{table} 

To enumerate neighborly polytopes, we need to decide the realizability of these oriented matroids. We have performed some realizability tests (see Section~\ref{sec:realizability}) but we are still not able to completely classify these matroids among realizable and non-realizable. We can certify realizability for those polytopes obtained by sewing and Gale-sewing~\cite{P13}, and we have non-realizability certificates by biquadratic final polynomials~\cite{BR90}. Moreover, certain cases can be decided by studying their universal edges~\cite{RS91}. In particular, we are able to completely classify:

\begin{enumerate}[resume]
 \item all possible combinatorial types of neighborly $8$-polytopes with $12$ vertices.
\end{enumerate}

For the remaining cases, we only have upper and lower bounds. These results are summarized in Table~\ref{tb:summarypolytopes}, that shows upper and lower bounds for the 
 number of combinatorial types of (simplicial) neighborly polytopes (recall that every simplicial $d$-polytope with $n$ vertices corresponds to the face lattice of a uniform oriented matroid in $\om{d+1}{n}$). There is still a huge gap between the current upper and lower bounds for these numbers, although we expect the actual number to be closer to the upper bounds.
\begin{table}[htpb]
\begin{center}
{
\renewcommand{\arraystretch}{1.5}
\begin{footnotesize}
\begin{tabular}{c | c | c | c | c | c | c | c | c | c |}
        & n = 5 & n = 6 & n = 7 & n = 8 & n = 9 & n = 10 & n = 11& n = 12\\
 \hline
 d = 4 (r = 5) &  1    &   1    &   1   &    3    &  23    &  431 & \textbf{$\geq$ 3\,614}   & \textbf{$\leq$ 556\,144} \\
               &       &        &       &         &        &      &  \textbf{$\leq$ 13\,935}  &                         \\
 \hline
 d = 5 (r = 6) &       &    1   &    1    &   2   &    {126}   & \textbf{$\geq$ 8\,231}   & unk. & unk.\\
               &       &        &         &       &            &  \textbf{$\leq$ 159\,750} &      &      \\
 \hline
 d = 6 (r = 7) &       &       &   1     &   1  &   1  &  37  & \textbf{$\geq$ 11\,165}    &  unk. \\
               &       &       &         &      &      &      & \textbf{$\leq$ 42\,099}     &       \\
 \hline
 d = 7 (r = 8) &       &       &        &  1  &   1   &  4   &  \textbf{$\geq$ 35\,930} &  unk.\\
               &       &       &        &     &       &      &  \textbf{$\leq$ 35\,993} &      \\
 \hline
  d = 8 (r = 9) &       &       &        &    &   1   &  1   & 1  &  \textbf{ 2\,586 } \\

 \hline
\end{tabular}
\end{footnotesize}
}
\end{center}
\caption{The numbers of the combinatorial types of neighborly simplicial $d$-polytopes with $n$ vertices. Boldface results are new.}
\label{tb:summarypolytopes}
\end{table}

Among the direct consequences that can be deduced from this database there are the following results:
\begin{enumerate}[resume]
\item There are uniform neighborly oriented matroids without universal edges in $\om{5}{11}$, $\om{5}{12}$, $\om{7}{11}$ and $\om{9}{12}$ (only one such example, in $\om{5}{10}$, was known~\cite{BSt87}).
The latter (together with their realizability) gives a positive answer to a question by Richter and Sturmfels~\cite{RS91} concerning the existence of 
neighborly $2k$-polytopes with $2k+ 4$ vertices without universal edges.
\item There is a simplicial $5$-polytope with $9$ vertices that is not a quotient of any neighborly $8$-polytope with $12$ vertices (however, every simplicial $d$-polytope with $d+4$ vertices is a quotient of a neighborly $(2d+4)$-polytope with $(2d+8)$ vertices~\cite{K97}).
\item There are no oriented matroids ${\cal M}$ in $\om{5}{12}$ or $\om{7}{12}$ with an element $e$ such that ${\cal M}\setminus e$ is neighborly and ${\cal M}^*\setminus e$ is also neighborly. This implies that the equality case of the \emph{affine generalized upper bound conjecture}, if true, is probably not sharp for the $\leq \frac{n-d-1}{2}$ levels of a $d$-dimensional arrangement of $n$ affine halfspaces (this provides a negative answer to a question of Wagner, see~\cite{W06}).
\item The determinants of edge-valence matrices do not separate realizable neighborly oriented matroids and non-realizable neighborly oriented matroids (this implies that an observation by Bokowski and Shemer~\cite{BS87} does not hold in general).
\item Each of our simplicial neighborly $d$-polytopes with $n$ vertices satisfies property $S^*(\ffloor{n-d+1}{2}-1)$ (considered in \cite{HPT08}) but not necessarily property $S^*(\ffloor{n-d+1}{2})$ (it is remarked in \cite{HPT08} that the cyclic $(n-2k+1)$-polytope with $n$ vertices satisfies property $S^*(k)$ when $n$ is odd).
\end{enumerate}

Moreover, our results provide more evidence for the following problems and conjectures:
\begin{enumerate}[resume]
 \item Every stacked matroid polytope in $\om{4}{\leq 11}$ is a vertex figure of a rank $5$ neighborly oriented matroid (in \cite[Problem~1]{AS73}, 
 Altshuler and Steinberg asked if every stacked $3$-polytope is a vertex figure of a neighborly polytope).
 \item Every $2$-stacked ($2$-)neighborly $5$-polytope with $9$ vertices is a vertex figure of a ($3$-)neighborly $6$-polytope (Bokowski and Shemer~\cite{BS87} generalized Altshuler and Steinberg's question to whether
 every $(m-1)$-stacked $(m-1)$-neighborly $(2m-1)$-polytope is a vertex figure of an $m$-neighborly $2m$-polytope).
 \item Each uniform neighborly oriented matroid in $\om{r}{n}$ for which we have enumerated $\om{r}{n+1}$ can be extended to a neighborly oriented  matroid in $\om{r}{n+1}$ (whether this holds for all neighborly polytopes was asked in~\cite{S82}). 
 \item Each of the duals of our simplicial neighborly polytopes has a Hamiltonian circuit (this was proved for cyclic polytopes by Klee in \cite{K66}, where he suggested the study of the even-dimensional neighborly polytopes with this property).
 \item Each of the duals of our simplicial neighborly polytopes satisfies the Hirsch conjecture (this was known for these combinations of rank and corank~\cite{BDHS13}, but we make a complete classification of the oriented matroids in terms of their facet-ridge graph diameter).
\end{enumerate}

\section{Preliminaries and Notation}\label{preliminaries}
In this section, we provide basic definitions and notation on oriented matroids. For a comprehensive introduction to oriented matroids, see~\cite{OM}. 
We use the notation $[n]$ to denote the set $\{ 1, 2, \dots, n\}$ for $n \in \mathbb{N}$.
\subsection{Oriented matroids}
Let $E$ be a finite set.  An element of $\{ +,-,0\}^E$ is called a \emph{sign vector} on $E$; and if $X$ is a sign vector then $X^s$ denotes the set $\{ e \in E \mid X_e = s\}$ for $s \in \{ +,-,0 \}$.
 The \emph{composition} of two sign vectors $X$ and $Y$ is the sign vector $X \circ Y \in \{ +,-,0 \}^E$ such that
\[
(X \circ Y)_e = 
\begin{cases}
X_e & \text{if $X_e \neq 0$,} \\
Y_e & \text{otherwise}
\end{cases}
\]
for all $e \in E$.
Given two sign vectors $X$ and $Y$, their \emph{separation set} $S(X,Y)$ is 
\[ S(X,Y) = \{ e \in E \mid X_e = - Y_e  \neq 0 \}.\]
The \emph{reorientation} of a sign vector $X$ on a subset $A\subseteq E$ is the sign vector $_{-A}X$ fulfilling $(_{-A}X)^+ = (X^+ \setminus A) \cup (X^- \cap A)$ and $(_{-A}X)^- = (X^- \setminus A) \cup (X^+ \cap A)$.

We are ready to define \emph{oriented matroids} (by their covector axioms, see~\cite{OM} for other axiomatics).
\begin{defn}[Covector axioms]
An \emph{oriented matroid} on the ground set $E$ is a pair ${\cal M}=(E,{\cal V}^*)$, where ${\cal V}^* \subseteq \{ +,-,0\}^E$  --- the set of \emph{covectors} of $\cal M$ --- satisfies the following axioms:
\begin{itemize}
\item[(L1)] ${\bm 0} \in {\cal V}^*$,
\item[(L2)] $X \in {\cal V}^*$ implies $-X \in {\cal V}^*$,
\item[(L3)] $X,Y \in {\cal V}^*$ implies $X\circ Y\in {\cal V}^*$,
\item[(L4)] if $X,Y \in {\cal V}^*$ and $e \in S(X,Y)$ then there exists $Z \in {\cal V}^*$
such that $Z_e = 0$ and $Z_f=(X\circ Y)_f=(Y\circ X)_f$ for all $f\notin S(X,Y)$. 
\end{itemize}
\end{defn}

The set of all covectors of ${\cal M}$ is denoted by ${\cal V}^*({\cal M})$ and 
admits a natural partial order
\[
X \preceq Y \Leftrightarrow \text{for all $e \in E$, $X_e = Y_e$ or $X_e = 0$}\ (\Leftrightarrow X\circ Y=Y).
\] 
The poset $({\cal V}^*({\cal M})\cup {\bm 1}, \preceq)$, where ${\bm 1}$ is a top element, is a lattice, called the \emph{big face lattice} of ${\cal M}$. The minimal non-zero elements of ${\cal V}^*({\cal M})$
are called the \emph{cocircuits} of $\cal M$, and denoted by ${\cal C}^*({\cal M})$. Every non-zero covector can be written as $V = C_1 \circ \dots \circ C_m$ for some $C_1,\dots,C_m \in {\cal C}^*({\cal M})$.

The \emph{rank} of ${\cal M}$, denoted $\rank({\cal M})$, is the rank of its underlying matroid. It coincides with $\rho ({\bm 1})-1$, where $\rho$ is the rank function of the big face lattice of ${\cal M}$. 
For a covector $X \in {\cal V}^*$, we define its \emph{rank} as $\rank_{\cal M}(X):=\rho({\bm 1})-\rho (X)$.
An oriented matroid ${\cal M}$ is \emph{uniform} if the underlying matroid is uniform; equivalently, if $|X^0| = \rank_{{\cal M}}(X) - 1$ for every covector $X \in {\cal V}^*({\cal M})$.

\subsection{Realizable oriented matroids}
To each vector configuration $W=(w_1,\dots,w_n)$ in $d$-dimensional Euclidean space, we can naturally associate the oriented matroid ${\cal M}_W = ([n],{\cal V}^*_W)$
of rank $d$, where
\[ {\cal V}^*_{W} := \{ ({\rm sign} (c^Tw_1), \dots, {\rm sign} (c^Tw_n)) \mid c \in \mathbb{R}^d \}. \]
We say that an oriented matroid ${\cal M}$ is \emph{realizable} if there is a vector configuration $W$ such that 
${\cal V}^*({\cal M}) = {\cal V}^*_W$, which we call a realization of $\cal M$. %

For a point configuration $P=(p_1,\dots,p_n)$ in $d$-dimensional Euclidean space,
its \emph{associated vector configuration} $V_P=(v_1,\dots,v_n)$ consists of the homogenized vectors
$v_i := (p_i, 1)$ in ($d+1$)-dimensional Euclidean space and its \emph{associated oriented matroid} is
${\cal M}_P = ([n], {\cal V}_{V_P})$.
Note that ${\cal M}_P$ contains the all-positive covector.
The point configuration $P$ is in convex position if and only if for each $e \in [n]$ the covector $X_e$
with $(X_e)^0 = \{ e \}$ and $(X_e)^+ = [n] \setminus \{ e \}$ is a covector of ${\cal M}_P$.
Moreover, the convex hull of $F= \{ p_{i_1},\dots,p_{i_k} \}$ is a face of $\conv(P)$
if and only if ${\cal M}_P$ has the covector $X_F$ 
such that $(X_F)^+ = [n] \setminus \{ i_1,\dots,i_k\}$ and $(X_F)^0 = \{ i_1,\dots,i_k \}$.
Finally, the point configuration is in general position if and only if the associated matroid is uniform. 
In particular, simplicial polytopes always have at least one realization with a uniform associated oriented matroid.

The following definitions are motivated by these observations.

\begin{defn}%
An oriented matroid is \emph{acyclic} if it contains the all-positive covector.
\end{defn}

\begin{defn}%
Let ${\cal M}$ be an acyclic oriented matroid of rank $r$ on a ground set $E$.
A set $F \subseteq E$ is a \emph{face} of ${\cal M}$ if there is the covector $X_F$ of ${\cal M}$ such that $(X_F)^+ = E \setminus F$ and $(X_F)^0 = F$.
The poset formed by all faces of an oriented matroid ${\cal M}$ is called the \emph{face lattice} (or \emph{Las Vergnas lattice}) of ${\cal M}$.
\end{defn}

\begin{defn}
 An acyclic oriented matroid on a ground set $E$ is called a \emph{matroid polytope}\footnote{This concept should not be confused with the \emph{matroid basis polytope}, the convex hull of the indicator vectors of bases of a matroid, which is sometimes also called a \emph{matroid polytope}.} if every element is a face; that is, 
if for every $e \in E$ it has the covector $X_e$
with $(X_e)^0 = \{ e \}$ and $(X_e)^+ = E \setminus \{ e \}$.
\end{defn}

An oriented matroid can be realized by (the associated vector configuration of) some point configuration (resp.\ point configuration in convex position)
if and only if it is a realizable acyclic oriented matroid (resp.\ realizable matroid polytope). In that case, the face lattice of the convex hull of the point configuration coincides with the Las Vergnas lattice of the oriented matroid.
\subsection{Basic operations for oriented matroids}

Let ${\cal M} = (E,{\cal V}^*)$ be an oriented matroid.
For $F \subseteq E$,
\[ {\cal M} \setminus F := (E \setminus F, {\cal V}^*|_{E \setminus F}),
\text{ where ${\cal V}^*|_{E \setminus F} := \{ X|_{E \setminus F} \mid X \in {\cal V}^*\}$} \] 
is also an oriented matroid. It is called the \emph{deletion} of $F$ in ${\cal M}$.
The deletion ${\cal M} \setminus (E \setminus F)$ is also called the \emph{restriction} of ${\cal M}$ to $F$ and denoted by ${\cal M}|_F$.

Similarly, for $F \subseteq E$,
\[ {\cal M} / F := (E \setminus F, {\cal V}^*/F),
\text{ where ${\cal V}^*/F := \{ X|_{E \setminus F} \mid X \in {\cal V}^*, X^0 \supseteq F\}$} \] 
is an oriented matroid. It is called the \emph{contraction} of $F$ in ${\cal M}$.
If $P$ is a polytope and $v$ one of its vertices, then the associated oriented matroid of the vertex figure $P/v$ is the contraction ${\cal M}_{P/v}={\cal M}_{P}/v$; and if $F$ is a face of $P$, the associated oriented matroid of the quotient $P/F$ is ${\cal M}_{P/F}={\cal M}_{P}/F$.

If an oriented matroid ${\cal M}$ of rank $r$ on a ground set $E$ can be written as
${\widehat {\cal M}} |_E$ for some oriented matroid ${\widehat {\cal M}}$ of rank $r$ 
on a ground set $F$ such that $E \subseteq F$ and $|F \setminus E|=1$,
the oriented matroid ${\widehat {\cal M}}$ is said to be a \emph{single element extension} of ${\cal M}$.
An important observation is that for any $X \in C^*({\cal M})$, there is a unique sign $\sigma(X) \in \{ +,-,0\}$ satisfying $(X,\sigma (X)) \in C^*({\widehat {\cal M}})$.
The assignment $\sigma : C^*({\cal M}) \rightarrow \{ +,-,0\}$ determined in this way is called a \emph{localization}.
Another important observation is that
${\widehat {\cal M}}$ is uniquely determined by ${\cal M}$ and the localization $\sigma : C^*({\cal M}) \rightarrow \{ +,-,0\}$.
Therefore, all possible single element extensions of ${\cal M}$ can be enumerated by enumerating localizations.
For more details, see \cite[Section 7.1]{OM}.

\subsection{Isomorphisms of oriented matroids}
There are several kinds of natural isomorphisms for oriented matroids.
\begin{defn}[Equivalence relations for oriented matroids]\ 

\begin{itemize}
\item Two oriented matroids ${\cal M} = (E_{\cal M}, {\cal V}_{\cal M}^*)$ and ${\cal N} = (E_{\cal N}, {\cal V}_{\cal N}^*) $ 
are \emph{relabeling equivalent} if there is a bijection $\phi : E_{\cal M} \rightarrow E_{\cal N}$ such that 
$X \in {\cal V}_{\cal M}^* \Leftrightarrow \phi (X) \in {\cal V}_{\cal N}^*$.
\item Two oriented matroids ${\cal M} = (E_{\cal M}, {\cal V}_{\cal M}^*)$ and ${\cal N} = (E_{\cal N}, {\cal V}_{\cal N}^*) $ 
are \emph{reorientation equivalent}
 if there exists a subset $A \subseteq E_{\cal M}$ such that $_{-A}{\cal M}$ and ${\cal N}$ are relabeling equivalent,
  where $_{-A}{\cal M}$ is the \emph{reorientation} of ${\cal M}$ by $A$, i.e.,
  the oriented matroid with the set of covectors $\{ _{-A}X \mid X \in {\cal V}^*\}$.
\item Two matroid polytopes ${\cal M}$ and ${\cal N}$ have the same \emph{combinatorial type} if they have isomorphic (Las~Vergnas) face lattices.
\end{itemize}
\end{defn}

Throughout the paper, when we refer to numbers of oriented matroids, we consider them up to {relabeling equivalence}. Whenever we enumerate {combinatorial types} we state it explicitly.

\subsection{Neighborly oriented matroids}
Neighborly matroid polytopes are the oriented-matroid generalization of neighborly polytopes (see~\cite{St88} or \cite[Section~9.4]{OM}). 
\begin{defn}%
An oriented matroid ${\cal M}$ of rank $r$ on a ground set $E$ is 
\emph{$k$-neighborly} if every $k$-subset of $E$ is a face of ${\cal M}$.
A \emph{neighborly oriented matroid} or \emph{neighborly matroid polytope}\footnote{In the literature, and here, both terms \emph{neighborly oriented matroid} and \emph{neighborly matroid polytope} are used interchangeably. We only consider matroid polytopes, but is also natural to define acyclic matroids as $0$-neighborly oriented matroids.} is a matroid polytope that is $\ffloor{r-1}{2}$-neighborly.
\end{defn}
One of the outstanding properties of even-dimensional neighborly polytopes, first observed by Shemer in~\cite{S82}, 
is that they are \emph{rigid}. This was extended in \cite{St88} to neighborly oriented matroids of odd rank.
\begin{defn}
 An oriented matroid ${\cal M}$ is \emph{rigid} if it is uniquely determined by its face lattice, i.e., 
 any oriented matroid with the same face lattice coincides with ${\cal M}$.
\end{defn}

\begin{thm}[{{\cite{S82}\cite[Theorem~4.2]{St88}}}]
Neighborly oriented matroids of odd rank are {rigid}.
\end{thm}

Therefore, two neighborly oriented matroids of odd rank are relabeling equivalent if and only if
they have the same combinatorial type.

A final observation that we will allude to later.
\begin{lem}
[{{cf. \cite[Remark~9.4.10]{OM}}}]
\label{lem:uniform}
Neighborly oriented matroids of odd rank are always uniform. 
\end{lem}

\subsection{More notation}

Additionally, we use the following notation, where $\cal M$ is an oriented matroid, $E$ its ground set and $X$ is a covector:
\begin{multicols}{2}
\begin{itemize}[itemsep=0mm,leftmargin=*]
\item $\Lambda(n,m):=\{ (i_1,\dots,i_m) \mid 1\! \leq\! i_1\! <\! \dots\! <\! i_m\! \leq\! n\}$ for $m,n \in \mathbb{N}$.
\item $Z(X):= \{ e \in E \mid X_e = 0 \}$ $(= X^0)$.
\item $P_i({\cal M})$: non-negative covectors of $\cal M$ of rank~$i$.
\item $P({\cal M})$: non-negative cocircuits of $\cal M$.
\item $N({\cal M})$: non-positive cocircuits of $\cal M$.
\end{itemize}
\end{multicols}

\section{Algorithm for enumerating neighborly oriented matroids}\label{sec:computation}
\begin{goal}
 Given $r,n,k \in \mathbb{N}$ with $n \geq r+2$ and $k \leq \lfloor \frac{r-1}{2} \rfloor$, enumerate all possible rank $r$ uniform (realizable) $k$-neighborly oriented matroids 
on the ground set $[n]$, up to relabeling equivalence.
\end{goal}
 
Our approach to enumerate all uniform $k$-neighborly oriented matroids in $\om{r}{n}$ is incremental, and we always assume that all uniform $k$-neighborly oriented matroids in $\om{r}{n-1}$ are already available. The enumeration follows three steps:
\begin{description}
 \item[Step 1] For each uniform $k$-neighborly oriented matroid $\cal M$ in $\om{r}{n-1}$, we use a SAT solver to list functions 
 $\sigma: {\cal C}^*({\cal M}) \rightarrow \{ +,-\}$ that form a superset of all localizations of single element extensions of $\cal M$ 
 that are neighborly and uniform.
               In this step, it is important to reduce the number of candidates, i.e., to strengthen the SAT constraints.
 \item[Step 2] We compute which of the functions provide single element extensions.
 \item[Step 3] We compute one representative from each relabeling class.
\end{description}

\subsection*{Step 1. Enumerating candidates}

Let $r,n \in \mathbb{N}$ with $n \geq r+2$, and 
${\widehat {\cal M}}$ a rank $r$ uniform $k$-neighborly oriented matroid on $[n]$. 
Then ${\cal M}:={\widehat {\cal M}}|_{[n-1]}$ is a rank $r$ uniform $k$-neighborly oriented matroid on $[n-1]$.

Let $\sigma: {\cal C}^*({\cal M}) \rightarrow \{ +,-\}$ be the localization of the single element extension from $\cal M$ to $\widehat{\cal M}$. 
The faces of $\widehat{\cal M}$ are completely determined by $\sigma|_{P({\cal M})}$ and the faces of $\cal M$ by what is known as the 
\emph{beneath-beyond} method (see~\cite{G03,M81} and \cite[Proposition~9.2.2]{OM}). Namely, $Z$ is a rank $m$ non-negative covector of 
${\widehat {\cal M}}$ if and only if
\[
\begin{cases}
Z=(X, +) & \text{ for some $X \in P_m({\cal M})$ such that $\sigma (X) = +$}\\
 \text{ or}\\
Z=(Y, 0) & \text{ for some $Y \in P_{m-1}({\cal M})$ such that} \\
       & \text{ there exist $V,W \in P({\cal M})$ such that $V,W \preceq Y$ and $\sigma (V)=+, \sigma (W)=-$.}
\end{cases}
\]

The following lemma is straightforward from this characterization of the faces of $\widehat{\cal M}$.
\begin{lem}\label{lem:SATcondition:neighborly}
Let $\widehat{\cal M}$ be a $k$-neighborly uniform oriented matroid that is a single element extension of $\cal M$ with localization $\sigma$. Then, 
for all $m=1,\dots,k$ and $(i_1,\dots,i_m) \in \Lambda (n,m)$,
\[
\begin{cases}
\text{ $\sigma (C_1) = +$ or $\dots$ or $\sigma (C_{p})=+$} \\
\text{ or }  \\
\text{ ($\sigma (D_1) = +$  or $\dots$ or $\sigma (D_{q})=+$)} 
\text{ and } 
\text{ ($\sigma (D_1) = -$ or $\dots$ or $\sigma (D_{q})=-$),}
\end{cases}
\]
where $C_1,\dots,C_p$ are the non-negative cocircuits of ${\cal M}$ with $Z(C_j) \supseteq \{ i_1,\dots,i_m \}$ 
for $j=1,\dots,p$
and
$D_1,\dots,D_q$ are the non-negative cocircuits of ${\cal M}$ with $|Z(D_j) \cap \{ i_1,\dots, i_m \}| = m-1$
for $j=1,\dots,q$. (Notice that $\sigma(C)\neq0$ for every cocircuit $C$ because we only consider uniform oriented matroids.)
\end{lem}

Consider the simplicial complexes $P^+$ and $P^-$ generated respectively by $\{ Z(C)\ \mid\ C \in P({\cal M}),\ \sigma(C)=+\}$ and by $\{ Z(D)\ \mid\ D \in P({\cal M}),\ \sigma(D)=-\}$. Observe that $P^+$ is the antistar of $n+1$ in $P({\widehat {\cal M}})$. Analogously, $P^-$ is the antistar of $n+1$ in the face lattice of the extension of ${\cal M}$ with localization $-\sigma$. 

$P({\widehat {\cal M}})$ is always an $(r-2)$-dimensional PL-sphere (see \cite[Proposition 9.1.1]{OM}), the star of $n+1$ in $P({\widehat {\cal M}})$ is then a PL-ball (see \cite[Theorem 4.7.21]{OM}) and thus $P^\pm$ are PL-balls (see \cite[Corollary 3.13]{PL}). (These are actually lifting triangulations of oriented matroids, see \cite{S2002}.) 

In particular, $P^\pm$ is connected and, either it consists of a single simplex or for every simplex $S_1\in P^\pm$ there is a simplex $S_2\in P^\pm$ such that $S_1\cap S_2$ is a common facet of both $S_1$ and $S_2$. This implies the following lemma.

\begin{lem}\label{lem:SATcondition:connected}
 With the notation from above,  if $|\{ C \mid C \in P({\cal M}), \sigma(C)=+\} | > 1$ (it is required if $r \geq 5$) 
 and $\sigma(X)=+$ for $X \in P({\cal M})$,
then
\[ \sigma (C_1) = + \text{ or } \dots \text{ or } \sigma (C_{p}) = +, \]
where $C_1,\dots,C_{p}$ are the non-negative cocircuits of ${\cal M}$ such that $|Z(X) \cap Z(C_j)| = r-2$ for $j=1,\dots,p$.
The same holds for the simplicial complex generated by $\{ Z(D) \mid D \in P({\cal M}), \sigma(D)=-\}$.
\end{lem}

Moreover, the closed star of every face $F$ of $P^\pm$, $\st_F(P^\pm)$ is a PL-ball. 
Again, either $\st_F(P^\pm)$ consists of a single simplex or for every simplex $S_1\in \st_F(P^\pm)$ there is a simplex $S_2\in \st_F(P^\pm)$ such that $S_1\cap S_2$ is a common facet of both $S_1$ and $S_2$.
We obtain the following lemma as a consequence (compare with the discussion in \cite{BS87}).

\begin{lem}\label{lem:SATcondition:linksconnected}
With the notation from above, for all $X,Y \in P({\cal M})$ such that $|Z(X) \cap Z(Y)| < r-2$, 
\begin{equation*}\label{eq:shelling} \sigma(X)=\sigma(Y)=+  \Rightarrow (\sigma(C_{1})=+ \text{ or } \dots \text{ or } \sigma(C_{p})=+),\end{equation*}
where 
$C_{1},\dots,C_{p}$ are the cocircuits of ${\cal M}$ such that
$|Z(C_{j}) \cap Z(X)| = r-2$ and $Z(C_{j}) \supseteq Z(X) \cap Z(Y)$ for $j=1,\dots,p$.
Note that the same holds for the complex generated by $\{ Z(D) \mid D \in P({\cal M}), \sigma(D)=-\}$.
\end{lem}

For a fixed oriented matroid~$\cal M$, the enumeration of all the assignments $\sigma|_{P({\cal M})}: P({\cal M}) \rightarrow \{ +,-\}$ 
satisfying the above conditions (from Lemmas~\ref{lem:SATcondition:neighborly}, \ref{lem:SATcondition:connected} and 
\ref{lem:SATcondition:linksconnected}) is nothing but a SAT enumeration problem. Such a SAT enumeration problem can be solved by relsat~\cite{relsat} 
for example.

As a summary, in the first step of our enumeration, we solve the following problem: 
For each uniform $k$-neighborly matroid ${\cal M} \in \om{r}{n-1}$, 
enumerate all assignments $\phi : P({\cal M}) \rightarrow \{ \textsf{true}, \textsf{false}\}$ satisfying

\begin{itemize}
 \item  For all $m=1,\dots,k$ and $(i_1,\dots,i_m) \in \Lambda (n,m)$, if $C_1,\dots,C_p$ are the non-negative cocircuits of $P({\cal M})$ with $Z(C_j)\supseteq \{ i_1,\dots,i_m\}$ and 
$D_1,\dots,D_q$ are the non-negative cocircuits of  $P({\cal M})$ 
with $|Z(D_j) \cap \{i_1,\dots, i_m \}|= m-1$, then:
\begin{itemize}
 \item[$\blacktriangleright$]  $(\phi (C_1) \lor \dots \lor \phi (C_p))  \lor ((\phi (D_1) \lor \dots \lor \phi (D_q)) \land (\lnot \phi (D_1) \lor \dots \lor \lnot \phi (D_q))).$
 \begin{flushright} (Lemma \ref{lem:SATcondition:neighborly})\end{flushright}
\end{itemize}
\item For all $X \in P({\cal M})$, if $C_1,\dots,C_{p}$ are the non-negative cocircuits of ${\cal M}$ with $|Z(X) \cap Z(C_i)| = r-2$, then:
 
\begin{itemize}
 \item[$\blacktriangleright$] $\lnot \phi (X) \lor   (\phi (C_1) \lor \dots \lor \phi (C_{p}))$, \quad and
 \item[$\blacktriangleright$] $\phi (X) \lor   (\lnot \phi (C_1) \lor \dots \lor \lnot \phi (C_{p})).$  \begin{flushright}(Lemma \ref{lem:SATcondition:connected})\end{flushright}
\end{itemize}
\item For all $X,Y \in P({\cal M})$ such that $|Z(X) \cap Z(Y)| < r-2$, if $C_1,\dots,C_p$ are the cocircuits of ${\cal M}$ such that $|Z(C_i) \cap Z(X)| = r-2$ and $Z(C_{i}) \supseteq Z(X) \cap Z(Y)$, then:

\begin{itemize}
 \item[$\blacktriangleright$] $\lnot (\phi (X) \land \phi (Y) ) \lor (\phi (C_{1}) \lor \dots \lor \phi (C_{p})))$, \quad and
 \item[$\blacktriangleright$] $\lnot (\lnot \phi (X) \land \lnot \phi (Y) ) \lor (\lnot \phi (C_{1}) \lor \dots \lor \lnot \phi (C_{p}))).$
 \begin{flushright}(Lemma \ref{lem:SATcondition:linksconnected})\end{flushright}
\end{itemize}
\end{itemize}

\subsection*{Step 2. Computing compatible oriented matroids}
The next step is to enumerate all compatible localizations $\sigma$ with each $\sigma|_{P({\cal M})}$ determined by the SAT solutions
(or to compute one compatible localization 
if the resulting oriented matroid is known to be rigid in advance or 
if one wants to enumerate only face lattices of neighborly oriented matroids).
This can be done by checking whether compatible cocircuit signatures can be defined consistently along the \emph{colines} 
(i.e., rank $2$ contractions) of ${\cal M}$, based on Las Vergnas' characterization~\cite{LV78}. 
 We use a slightly modified version of the algorithm LocalizationsPatternBacktrack by Finschi and Fukuda~\cite{FF02} (starting with a specified $\sigma|_{P({\cal M})}$).
A similar algorithm to LocalizationsPatternBacktrack for uniform oriented matroids is also proposed by Bokowski and Guedes de Oliveira~\cite{BG00}.

The following is a characterization of localizations, due to Las Vergnas. 
It plays a fundamental role in our algorithm (and in LocalizationPatternBacktrack).
\begin{thm}~\mbox{}{\rm ({\cite[Theorem~7.1.8]{OM}\cite{LV78}})}\\
Let ${\cal M}$ be an oriented matroid and $\sigma :{\cal C}^* \rightarrow \{ +,-,0\}$
a cocircuit signature, satisfying $\sigma (-Y) = -\sigma(Y)$ for all $Y \in {\cal C}^*$.
Then the following statements are equivalent.
\begin{itemize}
\item[(1)] $\sigma$ is a localization: there exists a single element extension $\widehat{\cal M}$ of ${\cal M}$ such that
\[ \{ (Y,\sigma (Y)) \mid Y \in {\cal C}^* \} \subseteq \widehat{\cal C}^*,\]
where $\widehat{\cal C}^*$ is the set of cocircuits of $\widehat{\cal M}$.
\item[(2)] $\sigma$ defines a single element extension on every contraction of ${\cal M}$ of rank $2$.
\end{itemize}
\end{thm}

An overview of our algorithm is as follows.
For each SAT solution, the algorithm first determines the corresponding values of $\sigma|_{P({\cal M})}$, based on the SAT solution.
Then, for each coline of $\cal M$ we consider the cocircuits $C_1,\dots,C_{2k}$ ordered along it.
By Las Vergnas' characterization, there exists an index $m \in [2k]$ such that $\sigma(C_{m \bmod 2k}) = \dots = \sigma(C_{m+k-1 \bmod 2k}) = +$ and $\sigma(C_{m+k \bmod 2k}) = \dots = \sigma(C_{m+2k-1 \bmod 2k}) = -$. The cocircuit signature $\sigma$ is a localization of an oriented matroids if and only if the same holds for all colines.
All possible sign patterns that fulfill this condition are enumerated by backtracking.
That is, we first choose an index $m \in [2k]$ and substitute $\sigma(C_{m \bmod 2k}) = \dots = \sigma(C_{m+k-1 \bmod 2k}) = +$ and $\sigma(C_{m+k \bmod 2k}) = \dots = \sigma(C_{m+2k-1 \bmod 2k}) = -$.
If this conflicts with the values already assigned,  we abort the substitution and try other indices.
This backtracking procedure is repeated for all colines. Whenever the algorithm can continue this procedure until the last coline, then the obtained assignment $\sigma$ is a localization.
If this procedure does not complete any consistent assignment until the last coline, $\sigma|_{P({\cal M})}$ cannot extend to a localization.

\subsection*{Step 3. Computing relabeling classes}
Next, we have to classify oriented matroids up to relabeling equivalence.
It can be done efficiently by using the graph automorphism solver nauty~\cite{nauty}.
First, note that two uniform oriented matroids ${\cal M}$ and ${\cal N}$ are reorientation equivalent
if and only if their \emph{cocircuit graphs} are isomorphic~\cite{BFF01}.
Given an oriented matroid~${\cal M}$,
the cocircuit graph of ${\cal M}$ is the graph $CG({\cal M}) = (V_{CG}({\cal M}),E_{CG}({\cal M}))$ 
where $V_{CG}({\cal M}) = C^*({\cal M})$ and $e = \{ X,Y \} \in E_{CG}({\cal M})$ if and only if
$S(X,Y) = \emptyset$ and $Z \not\preceq X \circ Y$ for all $Z \in C^*({\cal M}) \setminus \{ X,Y \}$.
If ${\cal M}$ is uniform, we can say more simply that $\{X,Y\}$ is an edge if there are elements $e,f\in E$ such that $X_e=0\neq Y_e$ and $X_f\neq 0= Y_f$, and $X$ and $Y$ coincide on the remaining elements.
In the following, we modify cocircuit graphs in order to check relabeling equivalence.
Let $G({\cal M})=(V_{\cal M},E_{\cal M})$ be the graph defined as follows.
\begin{align*}
V_{\cal M} &= {\cal C}^*({\cal M}) \cup \{ v_+,v_-\}.\\
E_{\cal M} &= E_{CG}({\cal M}) \cup
\{ \{ v_+,g_+\} \mid g_+ \in P({\cal M})\} \cup \{ \{ v_-,g_-\} \mid g_- \in N({\cal M})\}.
\end{align*}

\begin{prop}
Uniform oriented matroids ${\cal M}$ and ${\cal N}$ are relabeling equivalent if and only if 
$G({\cal M})$ and $G({\cal N})$ are isomorphic.
\end{prop}
\begin{proof}
The only if part is trivial.
To prove the if part, consider two uniform oriented matroids ${\cal M}$ and ${\cal N}$ 
whose respective graphs $G({\cal M})$ and $G({\cal N})$ are isomorphic.
Let $\phi$ be an isomorphism between $G({\cal M})$ and $G({\cal N})$.
Clearly, it holds that $\phi \cdot {\cal M} = {\cal N}$ and $\phi (v_+) = v_+$ or $v_-$. 
By the result of \cite{BFF01}, ${\cal M}$ and ${\cal N}$ are reorientation equivalent.
Note that $\phi (P({\cal M})) = P({\cal M})$ or $\phi (P({\cal M})) =N({\cal M})$.
This leads to the relabeling equivalence.
\end{proof}

\section{Results}

Tables~\ref{tb:summaryneighborly} and~\ref{tb:summary2neighborly} display the number of neighborly (resp.\ $2$-neighborly) oriented matroids and the corresponding face lattices for the classes that we managed to enumerate. Table~\ref{tb:SATsolutions} complements this information with the number of SAT solutions obtained in the first step of our computation, as well as how many of these solutions admitted compatible oriented matroids.

\begin{table}[h]
\centering
\begin{footnotesize}
\begin{tabular}{|c|c|c|c|c|c|c|}
\hline
rank & \#elem. &   neigh. & \# SAT sols. & \# SAT sols. with    & \# non-isomorphic  & \# face lattices     \\
($r$) & ($n$)  &   ($k$)  &              & compatible OMs           &  compatible OMs     &   of OMs   \\
\hline
 5&8 & 2  &      107     &    105  &    3   &  3     \\
\hline
 5&9& 2  &      3\,266   &  2\,377  &    23  &  23     \\
\hline
 5&10& 2 &    330\,082   & 134\,554 &   432  &  432   \\
\hline
 5&11& 2 & 112\,442\,751  & 161\,597 & 13\,937 & 13\,937 \\
\hline 
 5&12& 2 & 74\,727\,217\,909 & 6\,735\,042 & 556\,144 & 556\,144 \\
\hline 
 6&9& 2  &    13\,862     &  11\,248 & 10\,825 &   126     \\
\hline
 6&10& 2 &  32\,054\,2731  & 8\,066\,523 & unk. & 159\,750 \\
\hline
 7&10& 3 &     6\,582    &  6\,582  &    37   &   37     \\
\hline
 7&11& 3 & 88\,234\,386   & 671\,352 & 42\,910 & 42\,910 \\
\hline
7&9&2   &      8\,324   &  8\,324   &   9    &    9   \\
\hline
7&10&2 &  5\,007\,497    & 1\,766\,098 &  unk.     &  4\,523     \\
\hline
 8&11& 3 & 14\,828\,993    & 9\,072\,880 & unk. & 35\,993 \\
\hline
8&10& 2  &    52\,535  &  46\,396   &   13    &    13   \\
\hline
8&11& 2  &  98\,660\,484 & 57\,512\,680  &  unk.     &  129\,968    \\
\hline
 9&12&4 & 1\,035\,430    & 669\,746 & 2\,592 & 2\,592 \\
\hline
\end{tabular}
\end{footnotesize}
\caption{Numbers of SAT solutions, compatible matroids and face lattices for $k$-neighborly oriented matroids of rank $r$ with $n$ elements.}
\label{tb:SATsolutions}
\end{table}

\subsection{Realizability}\label{sec:realizability}
The enumeration of combinatorially distinct face lattices of neighborly polytopes can be done by classifying neighborly oriented matroids into realizable and non-realizable. 
In this section we present a preliminary analysis, although the realizability of many neighborly matroids remains still undecided. The outcome of these results is shown in Table~\ref{tb:summarypolytopes}.

\subsubsection{Non-realizability certificates}\label{subsec:nonrealizability}

An efficient method to test non-realizability of oriented matroids is the \emph{biquadratic final polynomial (BFP)} method. 
This powerful technique produces non-realizability certificates based on linear programming relaxations of the Grassmann-Pl\"ucker relations.
It can decide all non-realizable oriented matroids in $\om{r}{n}$ for $(r,n)=(4,8),(3,9)$ and those that are uniform for $(r,n)=(3,10),(3,11)$~\cite{AAK02,AK07,BR90b,FMM13,R88}. 
The smallest known example of non-realizable oriented matroid with no BFP proof is in $\om{3}{14}$~\cite{R96}. 

We applied the BFP method to some classes of neighborly oriented matroids. We found a BFP non-realizability certificate for
\begin{itemize}[itemsep=0mm]
\item 1 out of 432 neighborly oriented matroids in $\om{5}{ 10}$,
\item 2 out of 13\,937 neighborly oriented matroids in $\om{5}{ 11}$, 
\item 811 of 42\,910 neighborly oriented matroids in $\om{7}{ 11}$, 
\item 6 of 2\,592 neighborly oriented matroids in $\om{9}{ 12}$. 
\item 63 %
neighborly oriented matroids in $\om{8}{11}$, where only one neighborly oriented matroid was picked up for each of the 35\,993 face lattices
      (therefore, it does not mean that these $63$ face lattices are non-polytopal).
\end{itemize}
The classes $\om{5}{12}$ and $\om{6}{10}$ were too large to perform this test.

Moreover, by oriented matroid duality, the realizability problem for uniform matroids in $\om{8}{11}$ is equivalent to that for uniform matroids in $\om{3}{11}$,
whose realizability is already completely classified~\cite{AK07}.
It is known that all non-realizable uniform oriented matroids in $\om{3}{11}$ can be detected by BFP.
Therefore, all the $35\,930$ neighborly oriented matroids in $\om{8}{11}$ that do not have a BFP non-realizability certificate are realizable.

\subsubsection{Realizability certificates}

The Gale sewing construction for neighborly polytopes, proposed in~\cite{P13}, provides a large family of neighborly oriented matroids. 
Gale sewing can be applied to any neighborly oriented matroid in $\om{r}{n}$ to obtain neighborly oriented matroids in $\om{r+2}{n+2}$.
Since this construction is based on lexicographic extensions, it always yields a realizable oriented matroid when it is applied to a realizable oriented matroid.

\begin{table}[h]
\centering\begin{footnotesize}
\begin{tabular}[t]{|c|c|}
\hline
		      & \# neighborly OMs \\
                      & (resp.\ face lattices) \\
(rank, \#elements)     & obtained by Gale sewing \\
                      &from a cyclic polytope \\
\hline
(5,8)  &      3    (3)       \\
\hline
(5,9)  &      18   (18) \\
\hline
(5,10) &    227  (227)  \\
\hline
(5,11) & 3\,614  (3\,614) \\
\hline
(6,9)  &  192 (47)      \\
\hline
\end{tabular}
\qquad\begin{tabular}[t]{|c|c|}
\hline
		      & \# neighborly OMs \\
                       & (resp.\ face lattices) \\
(rank, \#elements)     & obtained by Gale sewing \\
                      &from a cyclic polytope \\
\hline
(6,10)  &  52\,931 (8\,231)       \\
\hline
(7,10) &   28   (28)    \\
\hline
(7,11) & 9\,495  (9\,495)  \\
\hline
(9,12) & 975 (975)    \\
\hline
\end{tabular}

\end{footnotesize}
\end{table}

Note that all neighborly oriented matroids in $\om{5}{9}$ and $\om{7}{10}$ are realizable~\cite{AS73,BS87}.
Applying Gale sewing to them, we obtained $11\,165$ and $975$ neighborly polytopes in $\om{7}{11}$ and $\om{9}{12}$, respectively.
We did not compute the case $\om{5}{12}$ because of its size.

\subsubsection{Universal edges and $\om{9}{12}$}
\label{sec:ue_realizability}

In Section \ref{subsec:universal_edge} we classify our neighborly matroids according to their number of universal edges, which allows us to improve our study of realizability for neighborly matroids in $\om{9}{12}$.
An edge $\{e,e'\}$ of a neighborly oriented matroid ${\cal M}$ of odd rank is called \emph{universal} if the contraction ${\cal M}/\{e,e'\}$ is also neighborly. Universal edges correspond to \emph{inseparable} pairs of elements of the oriented matroid~\cite{RS91}.

In \cite{RS91} it is proved that every neighborly oriented matroid of rank $2k + 1$ with
$2k+ 4$ vertices that has at least $2k-4$ universal edges is realizable.
This implies that the $1\,968$ neighborly oriented matroids in $\om{9}{12}$ with at least $4$ universal edges are realizable.
Even more, the $2\,589$ neighborly oriented matroids in $\om{9}{12}$ with at least one universal edge are also easily
classified into realizable and non-realizable.
By the reduction sequence method in \cite{RS91}, realizability of these matroids is reduced to realizability
of rank $3$ oriented matroids with at most 11 elements, whose realizability is completely classified~\cite{AAK02,AK07,R88}.
As remarked in Section \ref{subsec:nonrealizability}, it is known that all non-realizable rank $3$ oriented matroids with up to $11$ elements
can be recognized by the BFP method. As a consequence, we know that $2\,583$ out of the $2\,589$ matroids in $\om{9}{12}$ with a universal edge are realizable and $6$ are non-realizable.

\subsubsection{Non-linear optimization}
\label{sec:ue_optimization}

With the BFP method we can decide realizability of all the neighborly matroids in $\om{9}{12}$ except for 
the three that do not have any universal edge (cf.~Section \ref{subsec:universal_edge}). These three matroids
are also realizable. We found realizations with the help of the \emph{SCIP Optimization Suite} \cite{SCIP}.
For example, the following three $8$-dimensional point configurations $\{p_i\}_{0\leq i\leq 11}$, $\{p'_i\}_{0\leq i\leq 11}$
and $\{p''_i\}_{0\leq i\leq 11}$ realize them:
\begin{scriptsize}
\begin{alignat*}{3}
p_0&=p'_0=p''_0=\mathbf{0};& p_i&=p'_i=p''_i=100 \cdot\mathbf{e}_i\qquad\text{ for }1\leq i\leq 8;\\
p_9&=(-6,-69,32,87,-13,100,6,61),\qquad&
p_{10}&=(99,-15,70,-95,-96,6,92,-6),\qquad &
p_{12}&=(24,-86,6,-17,-64,74,100,100);\\
p'_9&=(20,94,61,98,-10,-90,16,-75),&
p'_{10}&=(87,-3,-100,3,42,80,-97,-11),&
p'_{11}&=(-85,100,94,48,61,-100,99,-57);\\
p''_{9}&=(82,100,1,-83,-5,-42,33,-87),&
p''_{10}&=(72,32,36,-31,-72,-42,100,-30),&
p''_{11}&=(-100,-100,90,100,59,39,-28,72).
\end{alignat*}
\end{scriptsize}
This completes the enumeration of all neighborly $8$-polytopes with $12$ vertices.

\subsection{Quotients of neighborly polytopes}

\subsubsection{General quotients}
\label{sec:gquotient}

A longstanding open problem of Perles asks whether every combinatorial type of simplicial polytope appears as a quotient of an even dimensional neighborly polytope. This has been shown to hold for simplicial $d$-polytopes with $\leq d+4$ vertices \cite{K97}.

\begin{table}[h]
\centering\begin{footnotesize}
\begin{tabular}{|c|c|c|c|}
\hline
(rank,\#elements.)      & \# face lattices of      & \# face lattices of     & \# face lattices of \\
                      &   vertex figures         &   quotients by rank 2 subsets & quotients by rank 3 subsets \\
\hline
(6,10)                &      1\,137             &   23                  &    \\
\hline
(7,11)                &      23\,305              &   967               & 23 \\
\hline
(8,11)                &      4\,422               &   321               & 45 \\
\hline
(9,12)                &      5\,666               &   2\,408             & 298 \\
\hline
\end{tabular}
\end{footnotesize}
\end{table}

We computed the quotients of our neighborly oriented matroid polytopes and compared them with the list 
of combinatorial types of polytopes obtained in \cite{FMM13}. 
There are $322$ simplicial $5$-polytopes with $9$ vertices \cite{FMM13}. Among these, $321$ simplicial $5$-polytopes with $9$ vertices appear as quotients of neighborly oriented matroids
in $\om{8}{11}$. On the other hand, only $298$ simplicial $5$-polytopes with at most $9$ vertices appear as quotients of neighborly oriented matroids with $\om{9}{12}$.

The $5$-polytope $\PP$ with $9$ vertices whose facets are
\begin{footnotesize}
\begin{align*}
&[1 2 3 7 8] \ \ [1 2 3 7 9] \ \ [1 2 3 8 9] \ \ [1 2 5 7 9] \ \ [1 3 4 7 9] \ \ [1 2 5 8 9] \ \ [1 3 4 8 9] \ \ [2 3 6 7 9]  \ \
[2 3 6 8 9] \ \ [1 4 5 7 9] \ \ [1 4 5 8 9] \ \ [2 5 6 7 9] \\ 
&[3 4 6 7 9] \ \ [2 5 6 8 9] \ \ [3 4 6 8 9] \ \ [4 5 6 7 9] \ \ 
[4 5 6 8 9] \ \ [1 2 5 7 8] \ \ [1 3 4 7 8] \ \ [2 3 6 7 8] \ \ [1 4 5 7 8] \ \ [2 5 6 7 8] \ \ [3 4 6 7 8] \ \ [4 5 6 7 8]
\end{align*}
\end{footnotesize}
is not a quotient of a neighborly polytope whose oriented matroid is in $\om{7}{10}$, $\om{8}{11}$ or $\om{9}{12}$.
 (However, it is a quotient of neighborly polytopes in $\om{10}{13}$ and $\om{11}{14}$.)

This is indeed a face lattice of a realizable oriented matroid, and hence of a polytope. An affine Gale dual of a realization of this polytope is depicted in Figure~\ref{fig:gale}. Below we provide some geometric intuition for our claim.

 \begin{figure}[h]
 \begin{center}
 \includegraphics[scale=0.1, bb = 0 0 783 724, clip]{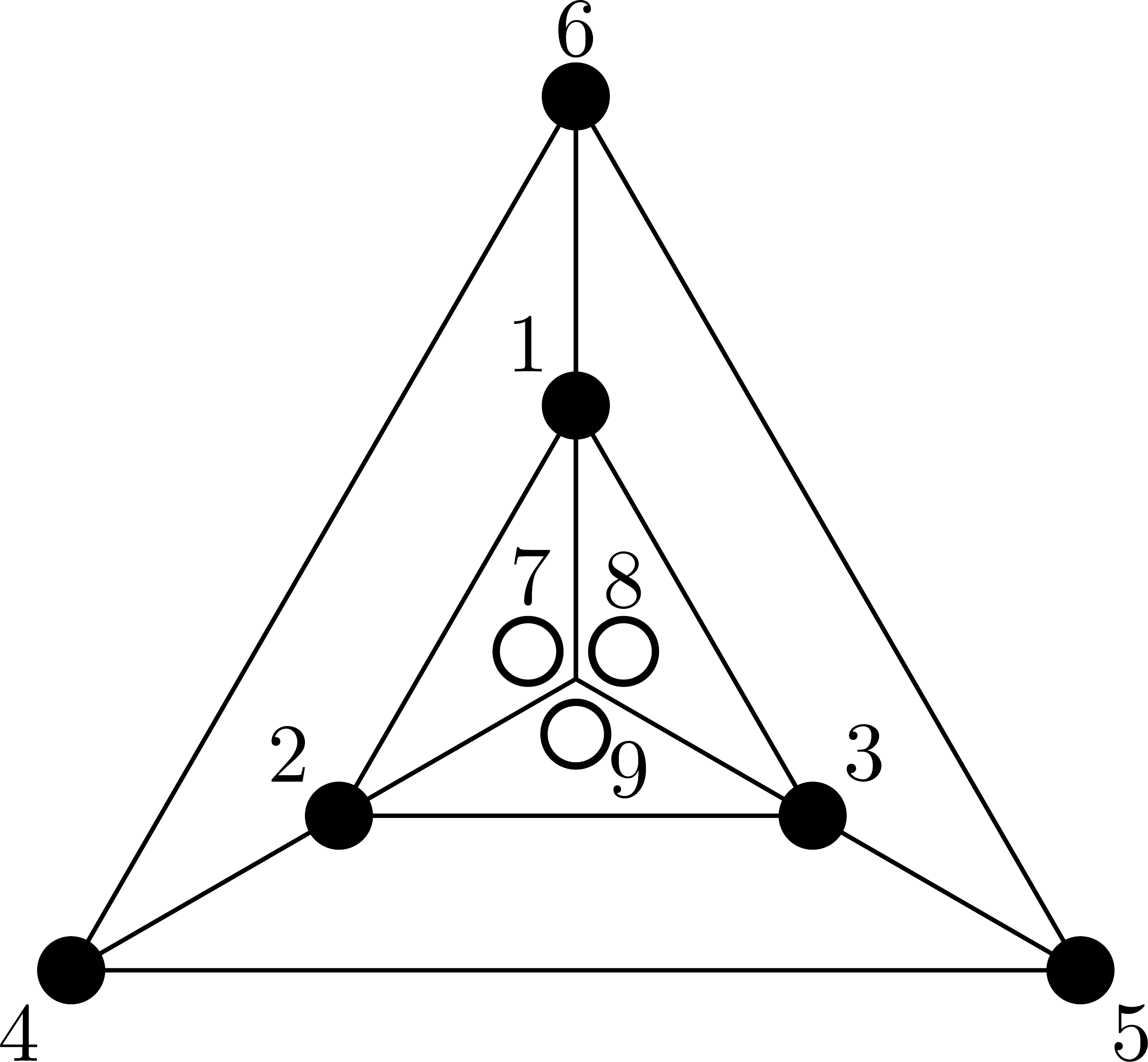}
 \caption{Affine Gale diagram of $\PP$.}
 \label{fig:gale}
 \end{center}
 \end{figure}
 
\begin{lem}
 Every neighborly polytope that has $\PP$ as a quotient must have at least $13$ vertices.
\end{lem}
\begin{proofidea}
We show that $12$ vertices is not enough. For this we use two properties of affine Gale duals (for which we refer to \cite{OM}), and argue on Figure~\ref{fig:gale} (although we argue with a concrete realization, it is easily seen that similar arguments carry on to any affine Gale diagram of $\PP$):
 \begin{enumerate}
  \item If $\PP$ is a quotient of $Q$, then its affine Gale dual is obtained by removing some points from that of~$Q$.
  \item\label{it:balanced} If $Q$ is a  neighborly polytope, then its affine Gale dual $A$ has the following property: 
  for every hyperplane $H$ spanned by $A$, the number of black points in $H^+$ minus the number of white points 
  in $H^+$ differs at most by one with the number of black points in $H^-$ minus the number of white points in $H^-$ (see \cite[Proposition~3.2]{St88} or \cite[Proposition~2.4]{P13}).
 \end{enumerate}
 Hence, we need to study how to add points to the Gale diagram of $\PP$ so that it fulfills property~\ref{it:balanced}.  Now, looking at the lines spanned by the points $1$, $2$ and $3$ we see that if we added two or three points, they should all be black and lie in the triangle spanned by $1$, $2$ and $3$. However, then the line spanned by $1$ and $4$ would still be unbalanced.
\end{proofidea}

Kortenkamp proved that every simplicial $d$-polytope with $d+4$ vertices is a quotient of a neighborly $(2d+4)$-polytope with $2d+8$ vertices~\cite{K97}.
It is easy to construct examples of simplicial $d$-polytopes with $d+4$ vertices that cannot be a quotient of any (even-dimensional) neighborly polytope
 with less than $2d+2$ vertices: any that has a missing edge.
This polytope is the first example where the trivial bound is not enough, and shows that Kortenkamp's result 
cannot be improved to a quotient of a neighborly $(2d-2)$-polytope with $2d+2$ vertices.

\medskip

We also observed that all $23$ simplicial $3$-polytopes with at most $8$ vertices appear as quotients of neighborly oriented matroids
in $\om{7}{11}$. Another observation is that the number of 
all simplicial matroid polytopes of rank $5$ with $9$ vertices ($1143$, 1 of which is non-polytopal)~\cite{ABS80} 
and that of the vertex figures of neighborly matroid polytopes of rank $6$ with $10$ vertices
differ by $6$.

\subsubsection{$(m-1)$-stacked $(m-1)$-neighborly polytopes}
A simplicial $d$-polytope is called \emph{$k$-stacked} if it admits a triangulation with no interior faces of dimension $d-k-1$ \cite{MN13}, 
and $1$-stacked polytopes are called simply \emph{stacked polytopes}. Every vertex figure of a neighborly $2m$-polytope with $n$ vertices is 
an $(m-1)$-stacked and $(m-1)$-neighborly $(2m-1)$-polytope with $n-1$ vertices. 
The converse statement, whether every combinatorial type $(m-1)$-stacked and $(m-1)$-neighborly $(2m-1)$-polytope appears as a vertex figure of a neighborly
polytope, 
has been investigated for small values of $n$ in the cases $m=2,3$ \cite{A77, AS73, BS87}, but no counterexample has been found yet. 
This problem was first posed by Altshuler and Steinberg for vertex figures of $4$-polytopes and stacked $3$-polytopes~\cite[Problem~1]{AS73}, 
and then extended to higher dimensional polytopes by Bokowski and Shemer~\cite{BS87}.

 \paragraph{Stacked $3$-polytopes:}
We study the oriented-matroid version of Altshuler and Steinberg's question, i.e.,
is every rank $4$ stacked matroid polytope a vertex figure of a neighborly matroid polytope of rank~$5$?
We generated all the possible combinatorial types of stacked $3$-polytopes with up to $11$ vertices and compared them with all the possible vertex figures (i.e., quotients by one element) of rank $5$ neighborly oriented matroids. The results are depicted in the following table (FL abbreviates face lattice).	
\begin{table}[h]
\centering
\begin{footnotesize}
\begin{tabular}{|c|c|c|}
\hline
$n$     & \# FL of vertex figures& \# FL of stacked  \\
        &  from $\om{5}{n+1}$      &  $3$-polytopes $n$ vertices \\
\hline
5  &      1   &   1  \\
\hline
6 &       1   &   1  \\
\hline
7 &       3   &   3  \\
\hline
\end{tabular}
\quad\begin{tabular}{|c|c|c|}
\hline
$n$     & \# FL of vertex figures& \# FL of stacked  \\
        &  from $\om{5}{n+1}$      &  $3$-polytopes $n$ vertices \\
\hline
9 &      24   &   24    \\
\hline
10 &     93   &   93 \\
\hline
11 &     434   &   434 \\
\hline
\end{tabular}
\end{footnotesize}
\end{table}

Observe that every stacked $3$-polytope with at most $11$ vertices appears as a vertex figure
of a rank~$5$ neighborly matroid polytope.
Among the 93 stacked $3$-polytopes with $10$ vertices, only $85$ appear as vertex figures of the neighborly matroid polytopes of rank~$5$ 
constructed by Gale sewing, which means that a hypothetical proof of the conjecture cannot be based on this construction.

\paragraph{$2$-stacked $2$-neighborly polytopes:}
In~\cite[Concluding remark (2)]{BS87}, Bokowski and Shemer posed the question of whether every simplicial $2$-stacked  $2$-neighborly $5$-polytope 
appears as a vertex figure of a neighborly $6$-polytope. 
There are $126$ simplicial ($2$-)neighborly $5$-polytopes with $9$ vertices.
We checked using TOPCOM~\cite{TOPCOM,TOPCOMpaper} which of them admitted a $2$-stacked triangulation. 
(If $m \leq \frac{d-2}{2}$, one can check $m$-stackedness of polytopes and spheres directly from their $f$-vectors 
by the Generalized Lower Bound Theorem~\cite{MN13}; but we deal with the case $m>\frac{d-2}{2}$.) 
It turned out that $55$ of them are $2$-stacked and hence they completely coincide with the vertex figures of $6$-polytopes with $10$ vertices.
This gives an affirmative answer to the first non-trivial case of the question posed by Bokowski and Shemer~\cite{BS87}, while the general question remains open. 
 
\subsection{Extendability}

An old open problem posed by Shemer in~\cite{S82} asks whether every neighborly
oriented matroid of odd rank can be extended to a neighborly oriented matroid.
That is, if for every ${\cal M}\in\om{2r+1}{n}$, there is a neighborly oriented
matroid ${\widehat {\cal M}}\in\om{2r+1}{n+1}$ such that ${\widehat {\cal M}}\setminus \{n+1\}={\cal
M}$.  We observed that every neighborly oriented matroid in $\om{5}{\leq 11} $
and $\om{7}{\leq 10}$ can be extended to a neighborly oriented matroid. 

\subsection{Facet-ridge graphs}
The {\it facet-ridge graph} of a polytope is that formed by connecting its facets that have common ridges.
It is a fundamental object of study in polytope theory \cite{BDHS13,BS11,K66,S12}. 

\subsubsection{Diameters}
The Hirsch conjecture stated that the facet-ridge graph diameter of a polytope $P$ was always not greater than $n-d$, 
where $n$ and $d$ are the number of vertices and the dimension of $P$ respectively. 
Although the conjecture has been disproved by Santos~\cite{S12}, the smallest counterexamples are still high dimensional. 
In particular, the Hirsch conjecture is still open in dimension~$4$. 
We compute all possible the diameters for neighborly matroid polytopes of small rank and corank, and state the number of instances that attain it. 
Note that the Hirsch conjecture is known to hold for all the classes of neighborly polytopes that we computed \cite{BDHS13,BS11}. 
For the duals of cyclic polytopes, the Hirsch conjecture is known to hold~\cite{K66}, and Kalai gave a polynomial upper bound $d^2(n-d)^2\log n$ of the diameters of the duals of neighborly $d$-polytope with $n$ vertices~\cite{K91}.
\begin{table}[h]
\centering
\begin{footnotesize}
\begin{tabular}[t]{|c|c|c|}
\hline
(rank, \#elements)     &    diameter[number of instances] \\
\hline
(5,9)    & 4 [22], 5 [1]  \\
\hline
(5,10)   &   4 [1], 5 [431]          \\
\hline
(5,11)  &      5 [7\,445], 6 [6\,492]     \\
\hline
(5,12)  &      5 [5], 6 [554\,374], 7 [1\,765]     \\
\hline
(6,10) &      4 [16], 5 [159,732]      \\
\hline
\end{tabular}
\quad\begin{tabular}[t]{|c|c|c|}
\hline
(rank, \#elements)     &    diameter[number of instances] \\
\hline
(7,10) &      4 [37]      \\
\hline
(7,11) &       5 [42\,910]     \\
\hline 
(8,11)  &      4 [35\,993]       \\
\hline
(9,12) &     4 [2\,592]        \\
\hline
\end{tabular}
\end{footnotesize}
\end{table}
\\
Moreover, all uniform $2$-neighborly oriented matroids in $\om{7}{10}$ and $\om{8}{11}$ were observed to have diameter $4$.
Based on these results, it is natural to ask whether the facet-graph of every $2$-neighborly uniform oriented matroid of corank $3$ with at least $8$ elements has diameter $4$ 
(it is always at most $4$ \cite{K64}).

\subsubsection{Hamiltonian circuits}
In \cite{K66}, Klee proved that the facet-ridge graph of a cyclic polytope always has a Hamiltonian circuit and asked whether
the same holds for every even-dimensional neighborly polyopes. 
We verified that this property is indeed shared by all our simplicial neighborly and $2$-neighborly polytopes (and matroid polytopes). 

\subsection{Number of universal edges}
\label{subsec:universal_edge}
Recall that a universal edge of a neighborly oriented matroid of odd rank is and edge whose contraction is still neighborly. 
Universal edges play an essential role in the construction techniques for neighborly polytopes~\cite{P13, S82} and are also important to check realizability of neighborly oriented matroids, as explained in Section~\ref{sec:ue_realizability}.
The number of universal edges of a neighborly oriented matroid of odd rank with $n$ elements is one of $0,1,\dots,n-2,n$ \cite{S82}. 
We enumerated our oriented matroids of odd rank according to their number of universal edges. The results are displayed in Table~\ref{tab:ue}.

\begin{table}[h]
\centering\begin{footnotesize}
\begin{tabular}{|c|c|}
\hline
(rank, \#elements)     & $m$ [\# neighborly oriented matroids  with $m$ universal edges]\\
\hline
(5,10) &    0 [1], 1 [2], 2 [17], 3 [85], 4 [159], 5 [114], 6 [40], 7 [11], 8 [2], 9 [0], 10 [1]   \\
\hline
(5,11) &   0 [20], 1 [213], 2 [1\,145], 3 [3\,463], 4 [4\,897], 5 [3\,040], 6 [965], 7 [170], 8 [21], 9 [2], 10 [0], 11 [1]  \\
\hline 
(5,12) & 0 [1\,454], 1 [14\,260], 2 [61\,870], 3 [148\,892], 4 [181\,944], 5 [108\,619],  \\
       & 6 [32\,997], 7 [5\,516], 8 [550], 9 [38], 10 [3], 11 [0], 12 [1] \\
\hline
(7,11) & 0 [221], 1 [2\,161], 2 [8\,479], 3 [14\,093], 4 [11\,298], 5 [5\,110], 6 [1\,329], 7 [195], 8 [21], 9 [2], 10 [0], 11 [1]    \\
\hline
(9,12) & 0 [3], 1 [33], 2 [169], 3 [419], 4 [695], 5 [655], 6 [402], 7 [165], 8 [40], 9 [8], 10 [2], 11 [0], 12 [1]     \\
\hline
\end{tabular}
\end{footnotesize}
\caption{Classification with respect to the number of universal edges.}\label{tab:ue}
\end{table}

So far, only one example of neighborly oriented matroid without any universal edge was known, in $\om{5}{10}$~\cite{BSt87}. We found new examples in the new classes. 
In particular, Richter and Sturmfels asked in \cite{RS91} whether there exists a neighborly $2k$-polytope with $2k+ 4$ vertices
without universal edges (Problem 5.5).
There are three neighborly oriented matroids in $\om{9}{12}$ that do not have any universal edges, and they are realizable
(see Section~\ref{sec:ue_optimization}), which answers this question in the affirmative.

\subsection{Edge-valence matrices}
Altshuler and Steinberg~\cite{AS73} defined \emph{edge-valence matrices} and
used their determinants as invariants of combinatorial types of neighborly polytopes.
The edge-valence matrix of a neighborly polytope $N$ with vertices $v_1,\dots,v_n$ 
is an $n \times n$ matrix $A=(a_{ij})$, where
$a_{ij}$ is the number of facets of $P$ that contain the edge $\{ v_i, v_j \}$
for $i,j=1,\dots,n$.

On the other hand, Bokowski and Shemer~\cite{BS87} used a modified version of 
edge-valence matrices, based on the notion of missing faces.
The convex hull $F$ of some vertices of a polytope $P$ is a \emph{missing face} of $P$
if $F$ is not a face of $P$. The notion of missing faces plays a fundamental role in neighborly polytope theory, see \cite{S82}.
They consider the matrix $M=(m_{ij})$, where
$m_{ij}$ is the number of missing faces of $P$ that contain $v_i$ and  $v_j$
for $i,j=1,\dots,n$.

We computed the determinants of these two versions of edge-valence matrices for our neighborly oriented matroids.

\begin{table}[h]
 \centering\begin{footnotesize}
\begin{tabular}{|c|c|c|c|}
\hline
(rank, \#elements)  &  {\pbox{5cm}{\raisebox{-.25cm}{\# FL of}\\neigh. OMs}} & \# $\det A$              & \# $\det M$  \\
\hline
(5,8)             & 3   &      3                    &  3        \\
\hline
(5,9)            &  23   &      23                   &   23       \\
\hline
(5,10)          &  432     &     429                  &  432        \\
\hline
(6,10)           & 159\,750     &     159\,364              &  32\,329      \\
\hline
(7,10)           & 37    &     37                  &  37           \\      
\hline
\end{tabular}\quad
\begin{tabular}{|c|c|c|c|}
\hline
(rank, \#elements)  &  {\pbox{5cm}{\raisebox{-.25cm}{\# FL of}\\neigh. OMs}} & \# $\det A$              & \# $\det M$  \\
\hline
(5,11)             &   32\,937 &    13\,936                &  13\,937         \\
\hline
(7,11)            & 42\,910    &     42\,903               &  42\,897         \\
\hline
(8,11)            & 35\,993    &      35\,993             &    523       \\
\hline
(5,12)            & 556\,144    &      556\,055              &    556\,141      \\
\hline
(9,12)            & 2\,592    &      2\,592              &     2\,588      \\
\hline
\end{tabular}
\end{footnotesize}
\end{table}

This computation shows that $\det A$ and $\det M$ do not completely distinguish combinatorial types
of neighborly (matroid) polytopes but that they are very powerful invariants.
In \cite{BS87}, the authors observed that the values of $\det M$ separate polytopes and non-polytopal spheres (in dimension $6$ with $10$ vertices).
In particular, they observed that $\det M$'s of non-polytopal spheres are greater than those of polytopal ones.
Combined with the results from Section~\ref{sec:realizability}, our computaion shows that the observation is not true any more in $\om{5}{10}$.

For further study, it might be interesting to study eigenvalues of $A$ and $M$.
In particular, it would be interesting to investigate when $\det A = 0$ holds.

\subsection{Sharp instances of the generalized upper bound conjecture}
Uli Wagner posed the following question (personal communication).
\begin{question}\label{que:wagner}
 Is there, for every possible rank and number of elements, a (realizable) oriented matroid~$\cal M$ with an element~$e$, such that ${\cal M}\setminus e$ is neighborly and ${\cal M}^*\setminus e$ is also neighborly.
\end{question}

The context of this problem is the search for sharp instances of the \emph{affine generalized upper bound conjecture (AGUBC)} \cite{W06}. 
For an arrangement ${\cal A}$ of affine halfspaces, let $v_{\leq \ell}({\cal A})$ be the number of vertices of ${\cal A}$ at level at most $\ell$ (cf. \cite{W06}). 
The AGUBC states that for every arrangement $\cal A$ of $n$ affine halfspaces in $\mathbb R^d$, 
$$v_{\leq \ell}({\cal A})\leq v_{\leq \ell}({\cal C}^*_ {n,d})$$
for $0\leq \ell\leq n-d$. Here, ${\cal C}^*_ {n,d}$ is the polar-to-cyclic spherical arrangement. 
Observe that for $\ell=0$ the AGUBC is just the upper bound theorem for convex polytopes.

In oriented matroid language, the conjecture says that the number of cocircuits $X$ with $|X^-|\leq \ell$ and $X_g=+$ in any (realizable) affine oriented matroid $({\cal M},g)$ (see \cite[Section 4.5]{OM}) is never larger than the number of cocircuits $X$ with $|X^-|\leq \ell$ of a uniform neighborly matroid.

The AGUBC is proved up to a factor of $2$ by Wagner~\cite{W06}. However, there are no lower bounds for $\max_{{\cal A}} v_{\leq \ell}({\cal A})$ when $v_{\leq \ell}$ is in the range $\ceil{n/(d+1)}\leq \ell \leq {(n-d-1)/2}$. 

Now, call a point configuration $S$ \emph{$\ell$-centered around $o$} if every affine hyperplane through $o$ contains at least $\ell-1$ points of $S$ in each of the open halfspaces that it defines.
Any neighborly $d$-polytope with $n$ vertices that is $\ell$-centered around a point $o\in \mathbb R^d$ 
provides a sharp instance of the AGUBC for $\leq\ell$-levels. Indeed, any polar-to-neighborly arrangement has the same number of elements at any level, and the $\ell$-centeredness condition ensures that all these vertices are affine (i.e. that they belong to the ``northern hemisphere'' polar to $o$, cf.~\cite{W06}).
For the particular case $\ell ={(n-d-1)/2}$, the centeredness condition is easily seen to be equivalent to being the Gale dual of a neighborly polytope (see, for example, \cite{P13}).

However, we found that there is only one oriented matroid in $\om{5}{10}$ that admits an extension fulfilling the condition of Question~\ref{que:wagner}
and none in $\om{5}{12}$ and $\om{7}{12}$. 

The equality case of the AGUBC is conjectured to hold only for polar-to-neighborly arrangements. Our results show that, in this case, the AGUBC would not always be sharp for the levels $\leq {(n-d-1)/2}$ of $d$-dimensional arrangements of $n$ affine halfspaces.

\subsection{Dual surrounding property}
In \cite{HPT08}, Holmsen, Pach and Tverberg studied \emph{$k$-surrounding sets}.
A finite point set $P \subseteq \mathbb{R}^d$ is said to possess the \emph{property $S(k)$} if for every $Q \subseteq P$ with $|Q|=k$ there exists an $R \subseteq P \setminus Q$ with $|R|=d+1-k$ such that $0 \in conv(Q \cup R)$. 

The Gale dual notion for $k$-surrounding sets is the \emph{property $S^*(k)$}.
A finite point set $\widetilde{P} \subseteq \mathbb{R}^{d}$ is said to possess the \emph{property $S^*(k)$}
if for every $\widetilde{Q} \subseteq \widetilde{P}$ with $|\widetilde{Q}|=n-k$, 
there exists an $\widetilde{R} \subseteq \widetilde{Q}$ with $|\widetilde{Q}|=d$ that forms a facet of $conv(\widetilde{P})$.
A point configuration $P \subseteq \mathbb{R}^d$ has the property $S(k)$ if and only if its Gale dual $P^* \subseteq \mathbb{R}^{n-d-1}$ fulfills the property $S^*(k)$.

In this context, the authors of \cite{HPT08} remark that the cyclic $(n-2k+1)$-polytope with $n$ vertices has property $S^*(k)$ when $n$ is odd. Motivated by this observation, we wanted to explore whether there is a stronger relation between neighborliness of $(n-2k+1)$-polytopes with $n$ vertices and the property $S^*(k)$.

\paragraph{Neighborliness does not imply $S^*(\ffloor{n-r+2}{2})$:}
We made some computational experiments to check for which values of $k$ do our neighborly oriented matroids possess the property $S^*(k)$.

\begin{table}[h]
\centering
\begin{footnotesize}
\vspace{0pt}
\begin{tabular}[t]{|c|c|c|}
\hline
(rank,\#elements)     &    $k$ [\# instances with property $S^*(k)$ \\
                      & but not property $S^*(k+1)$] \\
\hline
(5,8)    & 2 [3]  \\
\hline
(5,9)    & 2 [2], 3 [21]  \\
\hline
(5,10)   &   3 [432]          \\
\hline
(5,11)  &      3 [1\,533], 4 [12\,404]     \\
\hline
(5,12)  &       4 [556\,144]   \\
\hline
(6,9)  &      1 [19], 2 [107]     \\
\hline
\end{tabular}
\quad
\vspace{0pt}\begin{tabular}[t]{|c|c|c|}
\hline
(rank,\#elements)     &    $k$ [\# instances with property $S^*(k)$ \\
                      & but not property $S^*(k+1)$] \\
 \hline
(6,10) &       1 [424], 2 [158\,067], 3 [1\,259]    \\
\hline
(7,10) &       2 [37]     \\
\hline
(7,11) &       2 [6\,717], 3 [36\,193]     \\
\hline 
(8,11)  &      1 [1\,673], 2 [34\,320]       \\
\hline
(9,12) &     2 [2\,592]        \\
\hline
\end{tabular}\end{footnotesize}
\end{table}

We observe that simplicial neighborly matroid polytopes of rank $r$ on $n$ elements
do not always satisfy property $S^*(\ffloor{n-r+2}{2})$, not even when both $n$ and $r$ are odd.
On the other hand, those of odd rank in our database satisfy property $S^*(\ffloor{n-r}{2})$. Hence, 
it could be that 
this behavior holds in general.
Even more, it could be that if both $r$ and $n-r$ are odd, then 
every neighborly matroid polytope of rank $r$ on $n$ elements satsfies property $S^*(\ffloor{n-r+1}{2})$
but none does property $S^*(\ffloor{n-r+1}{2}+1)$.

\paragraph{$S^*(\ffloor{n-r+2}{2})$ does not imply neighborliness:} The rigid realizable uniform matroids of rank $3j-3$ and with $3j$ elements constructed by Sturmfels in \cite[Theorem 5.2]{St88} possess the property $S^*(2)$, yet none of them is neighborly. Is there such an example with $n-r$ even?

\section*{Acknowledgements}
We are indebted to Francisco Santos and Raman Sanyal for many interesting conversations, and
to Moritz Firsching for introducing us to \emph{SCIP} for the tests in Section~\ref{sec:ue_optimization}.
The first author is partially supported by JSPS Grant-in-Aid for Young Scientists (B) 26730002.
The research of the second author is supported by the DFG Collaborative
Research Center SFB/TR~109 ``Discretization in Geometry and Dynamics.''
Part of the experimental results in this research were obtained using supercomputing resources at Cyberscience Center, Tohoku University.

\bibliographystyle{amsplain}
\bibliography{enumeration}

\providecommand{\bysame}{\leavevmode\hbox to3em{\hrulefill}\thinspace}
\providecommand{\MR}{\relax\ifhmode\unskip\space\fi MR }
\providecommand{\MRhref}[2]{%
  \href{http://www.ams.org/mathscinet-getitem?mr=#1}{#2}
}
\providecommand{\href}[2]{#2}
\begin{thebibliography}{10}

\bibitem{SCIP}
Tobias Achterberg, \emph{{SCIP: Solving constraint integer programs}},
  Mathematical Programming Computation \textbf{1} (2009), no.~1, 1--41,
  \url{http://mpc.zib.de/index.php/MPC/article/view/4}.

\bibitem{AP14}
Karim~A. Adiprasito and Arnau Padrol, \emph{{The universality theorem for
  neighborly polytopes}}, Combinatorica (in press), preprint available at
  \href{http://arxiv.org/abs/1402.7207}{arXiv:1402.7207}.

\bibitem{AAK02}
Oswin Aichholzer, Franz Aurenhammer, and Hannes Krasser, \emph{{Enumerating
  order types for small point sets with applications}}, Order \textbf{19}
  (2002), no.~3, 265--281.

\bibitem{AK07}
Oswin Aichholzer and Hannes Krasser, \emph{{Abstract order type extension and
  new results on the rectilinear crossing number}}, Comput. Geom. \textbf{36}
  (2007), no.~1, 2--15.

\bibitem{A77}
Amos Altshuler, \emph{{Neighborly {$4$}-polytopes and neighborly combinatorial
  {$3$}-manifolds with ten vertices}}, Canad. J. Math. \textbf{29} (1977),
  no.~2, 400--420.

\bibitem{ABS80}
Amos Altshuler, J{\"u}rgen Bokowski, and Leon Steinberg, \emph{{The
  classification of simplicial {$3$}-spheres with nine vertices into polytopes
  and nonpolytopes}}, Discrete Math. \textbf{31} (1980), no.~2, 115--124.

\bibitem{AM73}
Amos Altshuler and Peter McMullen, \emph{{The number of simplicial neighbourly
  {$d$}-polytopes with {$d+3$} vertices}}, Mathematika \textbf{20} (1973),
  263--266.

\bibitem{AS73}
Amos Altshuler and Leon Steinberg, \emph{{Neighborly {$4$}-polytopes with {$9$}
  vertices}}, J. Combinatorial Theory Ser. A \textbf{15} (1973), 270--287.

\bibitem{AS85}
\bysame, \emph{{The complete enumeration of the {$4$}-polytopes and
  {$3$}-spheres with eight vertices}}, Pacific J. Math. \textbf{117} (1985),
  no.~1, 1--16.

\bibitem{BFF01}
Eric Babson, Lukas Finschi, and Komei Fukuda, \emph{{Cocircuit graphs and
  efficient orientation reconstruction in oriented matroids}}, European J.
  Combin. \textbf{22} (2001), no.~5, 587--600, Combinatorial geometries
  (Luminy, 1999).

\bibitem{relsat}
Roberto Bayardo, \emph{{relsat}}, {\tt http://code.google.com/p/relsat/}.

\bibitem{BW88}
Edward~A. Bender and Nicholas~C. Wormald, \emph{{The number of rooted convex
  polyhedra}}, Canad. Math. Bull. \textbf{31} (1988), no.~1, 99--102.

\bibitem{OM}
Anders Bj{\"o}rner, Michel {Las Vergnas}, Bernd Sturmfels, Neil White, and
  G{\"u}nter~M. Ziegler, \emph{{Oriented matroids}}, second ed., {Encyclopedia
  of Mathematics and its Applications}, vol.~46, Cambridge University Press,
  Cambridge, 1999.

\bibitem{BG00}
J{\"u}rgen Bokowski and Ant{\'o}nio~Guedes de~Oliveira, \emph{{On The
  Generation Of Oriented Matroids}}, Discrete Comput. Geom \textbf{24} (2000),
  555--602.

\bibitem{BG87}
J{\"u}rgen Bokowski and Klaus Garms, \emph{{Altshuler's sphere {$M^{10}_{425}$}
  is not polytopal}}, European J. Combin. \textbf{8} (1987), no.~3, 227--229.

\bibitem{BR90b}
J{\"u}rgen Bokowski and J{\"u}rgen Richter, \emph{{On the classification of
  non-realizable oriented matroids, Part {I}: Generation}}, Tech. Report Nr.
  1283, Technische Hochschule, Darmstadt, 1990.

\bibitem{BR90}
\bysame, \emph{{On the finding of final polynomials}}, European J. Combin.
  \textbf{11} (1990), no.~1, 21--34.

\bibitem{BS87}
J{\"u}rgen Bokowski and Ido Shemer, \emph{{Neighborly {$6$}-polytopes with
  {$10$} vertices}}, Israel J. Math. \textbf{58} (1987), no.~1, 103--124.

\bibitem{BSt87}
J{\"u}rgen Bokowski and Bernd Sturmfels, \emph{{Polytopal and nonpolytopal
  spheres: an algorithmic approach}}, Israel J. Math. \textbf{57} (1987),
  no.~3, 257--271.

\bibitem{BDHS13}
David Bremner, Antoine Deza, William Hua, and Lars Schewe, \emph{{More bounds
  on the diameters of convex polytopes}}, Optim. Methods Softw. \textbf{28}
  (2013), no.~3, 442--450.

\bibitem{BS11}
David Bremner and Lars Schewe, \emph{{Edge-graph diameter bounds for convex
  polytopes with few facets}}, Exp. Math. \textbf{20} (2011), no.~3, 229--237.

\bibitem{F14}
Wendy Finbow, \emph{{Simplicial neighbourly 5-polytopes with nine vertices}},
  Bolet{\'i}n de la Sociedad Matem{\'a}tica Mexicana (2014), 1--13.

\bibitem{wom}
Lukas Finschi, \emph{{Homepage of Oriented Matroids}}, {\tt
  http://www.om.math.ethz.ch/}.

\bibitem{FF02}
Lukas Finschi and Komei Fukuda, \emph{{Generation of oriented matroids---a
  graph theoretical approach}}, Discrete Comput. Geom. \textbf{27} (2002),
  no.~1, 117--136, Geometric combinatorics (San Francisco, CA/Davis, CA, 2000).

\bibitem{FF03}
\bysame, \emph{{Combinatorial generation of small point configurations and
  hyperplane arrangements}}, {Discrete and computational geometry}, {Algorithms
  Combin.}, vol.~25, Springer, Berlin, 2003, pp.~425--440.

\bibitem{FMM13}
Komei Fukuda, Hiroyuki Miyata, and Sonoko Moriyama, \emph{{Complete enumeration
  of small realizable oriented matroids}}, Discrete Comput. Geom. \textbf{49}
  (2013), no.~2, 359--381.

\bibitem{F06}
{\'E}ric Fusy, \emph{{Counting {$d$}-polytopes with {$d+3$} vertices}},
  Electron. J. Combin. \textbf{13} (2006), no.~1, Research Paper 23, 25.

\bibitem{GSL89}
G{\'e}rard Gonzalez-Sprinberg and Guy Laffaille, \emph{{Sur les arrangements
  simples de huit droites dans {${\bf R}{\rm P}^2$}}}, C. R. Acad. Sci. Paris
  S{\'e}r. I Math. \textbf{309} (1989), no.~6, 341--344.

\bibitem{GP80}
Jacob~E. Goodman and Richard Pollack, \emph{{On the combinatorial
  classification of nondegenerate configurations in the plane}}, J. Combin.
  Theory Ser. A \textbf{29} (1980), no.~2, 220--235.

\bibitem{G72}
Branko Gr{\"u}nbaum, \emph{{Arrangements and spreads}}, American Mathematical
  Society Providence, R.I., 1972, Conference Board of the Mathematical Sciences
  Regional Conference Series in Mathematics, No. 10.

\bibitem{G03}
\bysame, \emph{{Convex polytopes}}, second ed., {Graduate Texts in
  Mathematics}, vol. 221, Springer-Verlag, New York, 2003, Prepared and with a
  preface by Volker Kaibel, Victor Klee and G{\"u}nter M. Ziegler.

\bibitem{GS67}
Branko Gr{\"u}nbaum and V.~P. Sreedharan, \emph{{An enumeration of simplicial
  {$4$}-polytopes with {$8$} vertices}}, J. Combinatorial Theory \textbf{2}
  (1967), 437--465.

\bibitem{HPT08}
Andreas~F. Holmsen, J{\'a}nos Pach, and Helge Tverberg, \emph{{Points
  Surrounding the Origin}}, Combinatorica \textbf{28} (2008), no.~6, 633--644.

\bibitem{K91}
Gil Kalai, \emph{{The diameter of graphs of convex polytopes and {$f$}-vector
  theory}}, {Applied geometry and discrete mathematics}, {DIMACS Ser. Discrete
  Math. Theoret. Comput. Sci.}, vol.~4, Amer. Math. Soc., Providence, RI, 1991,
  pp.~387--411.

\bibitem{K64}
Victor Klee, \emph{{Diameters of polyhedral graphs}}, Canadian Journal of
  Mathmatics \textbf{16} (1964), 602--614.

\bibitem{K66}
\bysame, \emph{{Paths on polyhedra. {II}.}}, Pacific Journal of Mathematics
  \textbf{17} (1966), no.~2, 249--262.

\bibitem{K97}
Ulrich~H. Kortenkamp, \emph{{Every simplicial polytope with at most {$d+4$}
  vertices is a quotient of a neighborly polytope}}, Discrete Comput. Geom.
  \textbf{18} (1997), no.~4, 455--462.

\bibitem{LV78}
Michel {Las Vergnas}, \emph{{Extensions ponctuelles d'une g{\'e}om{\'e}trie
  combinatoire orient{\'e}e}}, {Probl{\`e}mes combinatoires et th{\'e}orie des
  graphes ({C}olloq. {I}nternat. {CNRS}, {U}niv. {O}rsay, {O}rsay, 1976)},
  {Colloq. Internat. CNRS}, vol. 260, CNRS, Paris, 1978, pp.~265--270.

\bibitem{MMIB12}
Yoshitake Matsumoto, Sonoko Moriyama, Hiroshi Imai, and David Bremner,
  \emph{{Matroid Enumeration for Incidence Geometry}}, Discrete Comput. Geom.
  \textbf{47} (2012), no.~1, 17--43.

\bibitem{nauty}
Brendan McKay, \emph{{nauty}}, {\tt http://cs.anu.edu.au/\~{}bdm/nauty/}.

\bibitem{M70}
Peter McMullen, \emph{{The maximum numbers of faces of a convex polytope.}},
  Mathematika, Lond. \textbf{17} (1970), 179--184.

\bibitem{M74}
\bysame, \emph{{The number of neighbourly {$d$}-polytopes with {$d+3$}
  vertices}}, Mathematika \textbf{21} (1974), 26--31.

\bibitem{M1988}
Nikolai~E. Mn{\"e}v, \emph{{The universality theorems on the classification
  problem of configuration varieties and convex polytopes varieties}},
  {Topology and geometry---{R}ohlin {S}eminar}, {Lecture Notes in Math.}, vol.
  1346, Springer, Berlin, 1988, pp.~527--543.

\bibitem{M81}
Beth~Spellman Munson, \emph{{Face lattices of oriented matroids}}, Ph.D.
  thesis, Cornell University, 1981.

\bibitem{MN13}
Satoshi Murai and Eran Nevo, \emph{{On the generalized lower bound conjecture
  for polytopes and spheres}}, Acta Math. \textbf{210} (2013), no.~1, 185--202.

\bibitem{P13}
Arnau Padrol, \emph{{Many {N}eighborly {P}olytopes and {O}riented {M}atroids}},
  Discrete Comput. Geom. \textbf{50} (2013), no.~4, 865--902.

\bibitem{TOPCOM}
J{\"o}rg Rambau, \emph{{TOPCOM}}, {\tt
  http://www.rambau.wm.uni-bayreuth.de/TOPCOM/}.

\bibitem{TOPCOMpaper}
\bysame, \emph{{{TOPCOM}: Triangulations of Point Configurations and Oriented
  Matroids}}, {Mathematical Software---ICMS 2002} (Arjeh~M. Cohen, Xiao-Shan
  Gao, and Nobuki Takayama, eds.), World Scientific, 2002, pp.~330--340.

\bibitem{R88}
J{\"u}rgen Richter, \emph{{Kombinatorische {R}ealisierbarkeitskriterien f{\"u}r
  orientierte {M}atroide}}, Mitt. Math. Sem. Giessen (1989), no.~194, 112.

\bibitem{RS91}
J{\"u}rgen Richter and Bernd Sturmfels, \emph{{On the Topology and Geometric
  Construction of Oriented Matroids and Convex Polytopes}}, Transactions of the
  American Mathematical Society \textbf{325} (1991), no.~1, 389--412.

\bibitem{R96}
J{\"u}rgen Richter-Gebert, \emph{{Two interesting oriented matroids}}, Doc.
  Math. \textbf{1} (1996), No. 07, 137--148.

\bibitem{RZ95}
J{\"u}rgen Richter-Gebert and G{\"u}nter~M. Ziegler, \emph{{Realization spaces
  of $4$-polytopes are universal}}, Bulletin of the American Mathematical
  Society \textbf{32} (1995), 403--412.

\bibitem{PL}
Colin~Patrick Rourke and Brian~Joseph Sanderson, \emph{{Introduction to
  piecewise-linear topology}}, {Springer Study Edition}, Springer-Verlag,
  Berlin, 1982, Reprint.

\bibitem{S2002}
Francisco Santos, \emph{{Triangulations of oriented matroids}}, Mem. Amer.
  Math. Soc. \textbf{156} (2002), no.~741, viii+80.

\bibitem{S12}
\bysame, \emph{{A counterexample to the {H}irsch conjecture}}, Ann. of Math.
  (2) \textbf{176} (2012), no.~1, 383--412.

\bibitem{S07}
Lars Schewe, \emph{{Satisfiability Problems in Discrete Geometry}}, Ph.D.
  thesis, Technische Universit{\"a}t Darmstadt, 2007.

\bibitem{S95}
Peter Schuchert, \emph{{Matroid-Polytope und Einbettungen kombinatorischer
  Mannigfaltigkeiten}}, Ph.D. thesis, TU Darmstadt, 1995.

\bibitem{S82}
Ido Shemer, \emph{{Neighborly polytopes}}, Israel J. Math. \textbf{43} (1982),
  no.~4, 291--314.

\bibitem{S1991}
Peter~W. Shor, \emph{{Stretchability of pseudolines is {NP}-hard}}, {Applied
  geometry and discrete mathematics}, {DIMACS Ser. Discrete Math. Theoret.
  Comput. Sci.}, vol.~4, Amer. Math. Soc., Providence, RI, 1991, pp.~531--554.

\bibitem{St88}
Bernd Sturmfels, \emph{{Neighborly Polytopes and Oriented Matroids}}, Eur. J.
  Comb. \textbf{9} (1988), no.~6, 537--546.

\bibitem{T62}
William~T. Tutte, \emph{{A new branch of enumerative graph theory}}, Bull.
  Amer. Math. Soc. \textbf{68} (1962), 500--504.

\bibitem{W06}
Uli Wagner, \emph{{On a Geometric Generalization of the Upper Bound Theorem.}},
  {FOCS}, 2006, pp.~635--645.

\end{thebibliography}

\end{document}